\documentclass{aims}
\usepackage{amsmath}
\usepackage{amssymb}
  \usepackage{paralist}
  \usepackage{graphics} 
  \usepackage{epsfig} 
\usepackage{graphicx}  \usepackage{epstopdf}
 \usepackage[colorlinks=true]{hyperref}
\hypersetup{urlcolor=blue, citecolor=red}

\def\currentvolume{X}
 \def\currentissue{X}
  \def\currentyear{200X}
   \def\currentmonth{XX}
    \def\ppages{X--XX}
     \def\DOI{10.3934/xx.xx.xx.xx}

\newtheorem{theorem}{Theorem}[section]

\newtheorem{lemma}[theorem]{Lemma}
\newtheorem{proposition}[theorem]{Proposition}

\newtheorem{problem}{Problem}
\theoremstyle{definition}
\newtheorem{definition}[theorem]{Definition}
\newtheorem{remark}[theorem]{Remark}

\newtheorem{assumption}{}

\title[Running heading with forty characters or less] 
      {Linear-Quadratic-Gaussian Mean-Field Controls of Social Optima}

\author[Zhenghong Qiu, Jianhui Huang and Tinghan Xie]{}


 \email{zhenghong.qiu@connect.polyu.hk}
 \email{james.huang@polyu.edu.hk}
 \email{tinghan.xie@connect.polyu.hk}

\thanks{The second author is supported by the financial support by RGC Grants PolyU 153005/14P, 153275/16P}

\thanks{$^*$ Corresponding author: tinghan.xie@connect.polyu.hk}

\begin{document}
\maketitle

\centerline{\scshape Zhenghong Qiu, Jianhui Huang and Tinghan Xie$^{*}$}
\medskip
{\footnotesize
 \centerline{Department of Applied Mathematics}
   \centerline{The Hong Kong Polytechnic University, Hong Kong, China}
} 

%

\bigskip

 \centerline{(Communicated by the associate editor name)}

\begin{abstract}
This paper investigates a class of unified stochastic linear quadratic Gaussian (LQG) social optima problems involving a large number of weakly-coupled interactive agents under a {generalized} setting. For each individual agent, the control and state process enters both diffusion and drift terms in its linear dynamics, and the control weight might be \emph{indefinite} in cost functional. This setup is {innovative and has great theoretical and realistic significance}
as its applications in mathematical finance {(e.g., portfolio selection in mean-variation model)}. Using some \emph{fully-coupled} variational analysis under person-by-person optimality principle, and mean-field approximation method, the decentralized social control is derived by a class of new type consistency condition (CC) system for typical representative agent. Such CC system is some mean-field forward-backward stochastic differential equation (MF-FBSDE) combined with \emph{embedding representation}. The well-posedness of such forward-backward stochastic differential equation (FBSDE) system is carefully examined. The related social asymptotic optimality is related to the convergence of the average of a series of weakly-coupled backward stochastic differential equation (BSDE). They are verified through some Lyapunov equations.
\end{abstract}
\section{Introduction}


Suppose that $(\Omega,\mathcal{F},\{\mathcal{F}_t\}_{0\leq t\leq T},\mathbb{P})$ is a complete filtered probability space, and $W(\cdot)$=$\{W_{1}(\cdot)$, $\cdots$, $W_N(\cdot)\}$ is a $N\times 1$-dimensional standard Brownian motion defined on it. $\{\mathcal{F}_t\}_{t\geq0}$ is the natural filtration generated by $W(\cdot)$ augmented by all $\mathbb{P}$-null sets in $\mathcal{F}$.
In this paper, we consider a weakly-coupled large-population system with $N$ agents, denoted by $\{\mathcal{A}^{i}\}_{1 \leq i \leq N}.$ The state process of the $i^{\text{th}}$ agent $\mathcal A^i$ is modeled by the following controlled linear stochastic differential equation (SDE) on finite time horizon $[0,T]$:
\begin{equation}\label{Eq1}
 \left\{
 \begin{aligned}
 & d{x}_i(t) = (A(t){x}_i(t) + B(t){u}_i(t) + F(t)x^{(N)}(t))dt\\
  &\hspace{1.5cm}+ (C(t){x}_i(t) + D(t){u}_i(t) + \tilde{F}(t)x^{(N)}(t))dW_i(t), \\
 & {x}_i(0) = \xi_0. \\
 \end{aligned}
 \right.
\end{equation}
where $x^{(N)}(t)\triangleq\frac{1}{N}{\sum_{i = 1}^{N}x_i(t)}$ denotes the {state-average}. $A(\cdot)$, $B(\cdot)$, $F(\cdot)$, $C(\cdot)$, $D(\cdot)$, $\tilde{F}(\cdot)$ are deterministic matrix-valued functions with appropriate dimensions. In addition, to evaluate the performance of our control, we also introduce the following individual cost functional for $\mathcal A^i$:
\begin{equation}\label{Eq2}
 \begin{aligned}
 \mathcal{J}_i(u(\cdot),\xi_0) =& \frac{1}{2}\mathbb{E}\bigg\{\int_{0}^{T}\|x_i(t) - 		\Gamma(t) x^{(N)}(t) - \eta(t)\|_{Q(t)}^2 + \|u_i(t)\|^2_{R(t)} dt\\
 &\hspace{1cm}+\|x_i(T) - 		\bar{\Gamma} x^{(N)}(T) - \bar{\eta}\|_{G}^2\bigg\},
 \end{aligned}
\end{equation}
where $u(\cdot)=(u_1(t),\cdots,u_N(t))\in\mathcal{U}_c$ and $Q(\cdot)$, $R(\cdot)$ and $G(\cdot)$ are weight matrices.
All agents work cooperatively and aim to minimize the social cost functional, which is denoted as follows:
\begin{equation}\label{Eq3}
\mathcal{J}^{(N)}_{soc}(u(\cdot)) = \sum_{i=1}^{N} \mathcal{J}_{i}(u(\cdot)).
\end{equation}
In what follows, for sake of notation simplicity, we may suppress the time subscript $t$ if necessary.\\

The above setup (1), (2) and (3) has the following features in its structure.\begin{itemize}
 \item (1) All agents are highly interactive and coupled in their dynamics and cost functionals due to the presence of state-average $x^{(N)}$. In particular, the individual cost functional depends on $u=\{u_{1}, u_{2}, \cdots, u_{n}\}$, the control profile of all agents owning to such weak-coupling. Thus, all agents frame a large-population system with mean-field type. Such system arises from various fields \cite{bfy2013}, \cite{cd2013}, \cite{tzb2014}, \cite{km2014}, \cite{ll2007}, \cite{mb2017}. Because of such mean-field structure, the related dynamic optimization is subject to the curse of dimensionality and complexity in numerical computation. Consequently, decentralized control which is based on the local
 information set will be used instead of centralized control which is based on the full information set. Note that the decentralized control for multi-agent system is well documented in literature (see \cite{cd2004}, \cite{hcm2007a}, \cite{lz2008} \cite{bj2017a}).

 \item (2) Unlike the noncooperative game in \cite{h2010}, \cite{mb2018}, \cite{nc2013}, where agents try to minimize their own individual cost functional, all agents in our model aim to minimize the social cost functional which is the summation of the cost functionals of all agents. The relevant analysis is also very different: all agents aim to reach the same criteria in Pareto or social optimality. They are seeking some social optimal points instead the Nash equilibrium. Also, in the presence of weakly-coupled interactive agents, some variational analysis should be applied to the person-by-person optimality to obtain some necessary condition for such Pareto optimality criteria. Note that the social optima is also well studied in literature, readers can refer to \cite{dr2012}, \cite{gmm2012}, \cite{tg1973},\cite{hch1972}, \cite{jm1955}, \cite{ra1962}.

 \item By (1) and (2), our work should be framed within the mean-field social optima with numerous cooperative agents. Note that such problem has drawn increasing research attentions in recent years, such as \cite{hcm2012}, \cite{bjj2017} for LQG social optima with constant noise, \cite{bj2017b} for related analysis in Markov LQG setup, \cite{hww2010} for related analysis in economic social welfare problem and \cite{bjj2017} for robust LQG social optima with drift uncertainty.

\end{itemize}

Note that in above state equation {a generalized setting is considered}. The state process $x_i(\cdot)$, control process $u_{i}(\cdot)$ and state-average $x^{(N)}$ enter the diffusion term when $C, D, \tilde{F} \neq 0$, { while in the recent research of mean-field social optima (see \cite{hcm2007a}, \cite{bjj2017}, \cite{wang2019social}), only constant volatility situation is studied.} In particular, when $D \neq 0,$ the diffusion term is dependent on the control directly and the related stochastic LQG problem can be referred as \emph{multiplicative-noise} control problem. The inclusion of control variable in diffusion is well motivated by various real applications. One such example comes from the well-known mean-variance portfolio problem (\cite{dr1991}, \cite{r1971}, \cite{xz2007}, \cite{zl2000}) where the control process (risky portfolio allocation) naturally enters the dynamics for given wealth process.
There are rich literature for the discussion of related LQG problem, readers can refer \cite{sly2016}, \cite{sy2014}, \cite{y2006}, \cite{y2013}, \cite{yz1999}.

Our control-dependent setup is different from \cite{jm2017}, \cite{h2010}, \cite{hcm2007a}, \cite{hcm2012}, \cite{bjj2017} in which $D=0$ and only drift is directly control-dependent. Such problem can be referred as \emph{additive-noise} LQG control problem, and has already been well investigated (see \cite{h2010}, \cite{hcm2007a}, \cite{mb2017}, \cite{bj2017b}). In principle, additive-noise problem has no essential distinction to deterministic or constant volatility LQG (see \cite{h2010}, \cite{hcm2012}, \cite{bjj2017}). Actually, with the help of constant variation method and separable property for linear system, the state can be represented by linear functional on state and control separately.

Besides, the weight matrices in cost functional are indefinite in our paper.  Recently, the LQG frameworks with indefinite weight matrices have been studied extensively in \cite{sly2016} and \cite{sy2014}. This setting has some interesting applications in mathematical finance (see \cite{yz1999}, \cite{zl2000}). However, in some recent literatures of team optimization (e.g., \cite{hcm2012}, \cite{bjj2017}, \cite{wang2019social}), only positive semi-definite weight is considered.  Consequently, our paper might  be the first time to formulate a team problem under {such generalized setting with realistic significance. In addition, such extension also brings some practical difficulties in our research.}

Compared with pervious works, the difficulties appear in our research are as follows:
\begin{enumerate}
  \item In contrast with the conditions in \cite{jm2017}, \cite{hcm2012}, \cite{mb2018}, \cite{bjj2017}, the weight coefficients in our model can be indefinite, and the convexity {of the social cost functional} need to be discussed. Unlike game problems, {where each agent only aims to minimize its own low-dimensional personal cost functional}, our team problem is a high-dimensional control problem essentially. Hence, it is very difficult to verify its convexity directly due to the dimensionality. However, the convexity is crucial to the problem solvability {(see \cite{sly2016})}. In this paper, we  bring some low-dimensional criteria for the convexity of the social cost functional by using some algebraic techniques.
  \item In the general frameworks of mean-field games  (see \cite{bfy2013}, \cite{bensoussan2016linear}, \cite{nie}), usually the auxiliary control problem can be obtained directly {by freezing the state-average as some deterministic term}. However, in the framework of social problem, this scheme will bring some ineffective strategy, which can not achieve the asymptotic optimality. Thus, variational techniques are applied to distinguish the high-order infinitesimals after the mean-filed approximation. In particular, $N +1$ additional adjoint processes should be introduced to deal with the cross-terms in the cost functional variation  and a new type of auxiliary control problem would be derived.
  \item {Our consistency condition (CC) system here is a highly coupled mean-field forward-backward stochastic differential equation (MF-FBSDEs) system, instead of an ordinary differential equations (ODEs) {system} in some other general cases. Due to $C, \tilde{F} \neq 0$, the adapted term enters the drift term of the CC system. Thus, the dynamics of the mean-field terms cannot be obtained  by taking expectation.}  It is complicated to investigate its solvability directly by decoupling method. Thus, by applying decentralizing method, we  transform it to a linear forward-backward stochastic differential equation (FBSDE) system, and study the well-posedness of {the} new system.
  \item Unlike some previous works in \cite{hcm2012}, \cite{bjj2017}, \cite{bj2017b}, we use linear operator method and Fr\'{e}chet derivative to prove the asymptotic social optimality. Because of $C, F, \tilde{F}\neq 0$, the error estimates are very hard to be given directly, since some of them are coupled without explicit expression. To overcome such difficulty, we decouple them via Lyapunov equations and estimate them in proper order.
\end{enumerate}

The main contributions of the paper are summarized as follows:
\begin{itemize}
 \item We setup a class of LQG control problem where both the drift and diffusion terms are dependent on {state process}, control process and state-average and give an approach to obtain the social optimal solution.
 \item By discussing the  weight coefficients, some low-dimensional criteria  for the convexity of the social cost functional, which is a high-dimensional system, are obtained.
 \item A highly coupled CC system (MF-FBSDE) is transformed to a equivalent FBSDE by decentralizing transformation. The existence and {uniqueness of the  CC system solution} is characterized by a Riccati equation.
 \item By using the perturbation { method to analyse} the Fr\'{e}chet derivative of cost functional, the decentralized mean-field controls we derived are proved to be asymptotically optimal. {To prove its asymptotical optimality, we  apply some classical estimates of SDEs, and investigate the optimality loss.}
\end{itemize}

This paper is organized as follows: In Section 2, we  formulates the social optimal LQG control problem. In Section 3, the convexity of social cost functional is discussed. We construct an auxiliary optimal control problem based on person-by-person optimality and design the decentralized control in Section 4 and Section 5 respectively. In Section 6, some prior lemmas are given, and based on them the asymptotic social optimality is proved. A numerical example is provided to simulate the efficiency of decentralized control in Section 7. Section 8 is the conclusion of this paper.

\section{Problem Formulation}

Before formulating the problem, we set a convention about notation.

\emph{Notation}: Throughout this paper, $\mathbb{R}^{n\times m}$ and $\mathbb{S}^n$ denote the set of all $(n\times m)$ real matrices and the set of all $(n\times n)$ symmetric matrices, respectively. We denote $\|\cdot\|$ as the standard Euclidean norm and $\langle\cdot,\cdot\rangle$ as the standard Euclidean inner product. For a vector $x$ and a symmetric matrix $S$, $\|x\|_S^2= \langle Sx,x\rangle=x^TSx$, where $x^T$ is the transpose of $x$. For a matrix $M$, the norm $\|M\|=\sqrt{\text{Tr}(M^TM)}$, and the max-norm, which is equivalent to the maximum absolute value of all elements, is denoted by $\|M\|_{\max}$. $M> 0 \ (\geq 0)$ means that $M$ is positive (semi-positive) definite and $M\gg 0$ means that, $M - \varepsilon I \geq 0$, for some $\varepsilon>0$. $\lambda_{\max}(M)$ denotes the maximum eigenvalue of matrix $M$, while $\lambda_{\min}(M)$ denotes the minimum eigenvalue of matrix $M$.
Let $\sigma$-algebra $\mathcal{F}_t^i=\sigma(W_i(s), 0\leq s\leq t)$ for $0\leq i\leq N$, $\mathcal{F}_t=\sigma(W_i(s), 0\leq s\leq t, 1\leq i\leq N)$,
 $\mathbb{F}^i=\{\mathcal{F}_t^i\}_{0\leq t\leq T}$, $0\leq i\leq N$ is the natural filtration generated by $W_i(\cdot)$. Correspondingly, we denote filtration
  $\mathbb{F}=\{\mathcal{F}_t\}_{0\leq t\leq T}$.
  Next, for Euclidean space $\mathbb{H}$, we introduce the following space:
\begin{equation*}
\begin{aligned}
L^{\infty}(0,T;\mathbb{H})=&\big\{\varphi : [0,T]\rightarrow \mathbb{H}\big| \ \varphi(\cdot) \text{ is bounded and measurable}\big\},\\
L_{\mathbb{F}}^2(\Omega;\mathbb{H})=&\big\{\xi : \Omega\rightarrow \mathbb{H}\big| \ \xi \text{ is }\mathbb{F}\text{-measurable, }\mathbb{E}\|\xi\|^2<\infty\big\},\\
L_{\mathbb{F}}^2(0,T;\mathbb{H})=&\big\{x : [0,T]\times\Omega\rightarrow \mathbb{H}\big| \ x(\cdot) \text{ is }\mathbb{F}\text{-progressively measurable, }\\
&\hspace{0.5cm}\|x(t)\|_{L^2}^2\triangleq\mathbb{E}\int_{0}^{T}\|x(t)\|^2dt<\infty\big\},\\
L_{\mathbb{F}}^2(\Omega;C([0,T];\mathbb{H}))=&\big\{x : [0,T]\times\Omega\rightarrow \mathbb{H}\big| \ x(\cdot) \text{ is }\mathbb{F}\text{-progressively measurable,} \\
&\hspace{0.5cm}\text{continuous, } \mathbb{E}\sup_{t\in[0,T]}\|x(t)\|^2<\infty\big\}.\\
\end{aligned}
\end{equation*}

 Further, base on the information structure, two types of admissible control sets are defined as follows. The centralized admissible control set is given by
\begin{equation*}
  \mathcal{U}_c = \Big\{(u_0,u_1,\cdots,u_N)|u_i(t)\in L_{\mathbb{F}}^2(0,T;\mathbb{R}^m), 0\leq i \leq N\Big\}.
\end{equation*}
Correspondingly, the decentralized admissible control set for $\mathcal A^i$ is given by
\begin{equation*}
\mathcal{U}_i=\left\{u_i|u_i(t)\in L_{\mathbb{F}^i}^2(0,T;\mathbb{R}^m)\right\},\quad 1\leq i \leq N.
\end{equation*}
Now, we introduce the following two assumptions
\begin{assumption}\label{A1}
 {\rm\textbf{(A1)}} The coefficients of the state equation satisfy the following:
 \[A, C, F, \tilde{F}\in L^\infty(0,T;\mathbb{R}^{n\times n}),\quad B, D \in L^\infty(0,T;\mathbb{R}^{n\times m}).\] 
\end{assumption}
\begin{assumption}\label{A2} {\rm\textbf{(A2)}}
The weighting coefficients in the cost functional satisfy the following:
\begin{equation*}\left\{
\begin{aligned}
&Q\in L^{\infty}(0,T;\mathbb{S}^n),\ \Gamma\in L^\infty(0,T;\mathbb{R}^{n\times n}),\ R\in L^{\infty}(0,T;\mathbb{S}^m),\ \eta\in L^{2}(0,T;\mathbb{R}^{n}),\\
&\bar{\Gamma}\in \mathbb{R}^{n\times n},\quad G\in\mathbb{S}^n,\ \bar{\eta}\in \mathbb{R}^{n}.
\end{aligned}\right.
\end{equation*}
\end{assumption}
 Under \textbf{(A\ref{A1})}, for any given $(u_1,\cdots,u_N)\in\mathcal{U}_c$, \eqref{Eq1} admits a unique solution $(x_1,\cdots,x_N)$.  Under \textbf{(A\ref{A2})}, the cost functional \eqref{Eq2} is well-posed. We aim to work out the optimal control for each agent to minimize the social cost functional. Thus we propose the following social optimal problem:
\begin{problem}\label{P1}
Minimize $\mathcal{J}^{(N)}_{soc}(u)$ over $\{u = (u_1,\cdots,u_N)|u_i\in\mathcal{U}_c\}$.
\end{problem}
However, in the practical application, usually, the agent can only access its own information (i.e., $\sigma\{\mathcal{F}_t^i \cup \sigma(x_i(s),s\leq t)\}$) and the information of the others may be unavailable for it. For this reason, we try to work out a so-called decentralized control, which only depends on the {agent's own} information, and we propose another problem.
\begin{problem}\label{P2}
Minimize $\mathcal{J}^{(N)}_{soc}(u)$ over $\{u = (u_1,\cdots,u_N)|u_i\in\mathcal{U}_i\}$.
\end{problem}

\section{Convexity}

{In this section, the convexity of the social cost functional will be studied. The weight coefficients $Q$, $R$ and $G$ profoundly influence the convexity of the cost functional. We start with positive semi-definite weight case, which is relatively simple, and then further consider indefinite weight case. Indefinite weight setting has some  interesting mathematical finance background (see \cite{yz1999}, \cite{zl2000}).}
Before that, the {state dynamics and the social cost functional} could be rewritten {in a high-dimensional form}.
\begin{equation}\label{Eq4}
 \begin{aligned}
 & d\mathbf{x} = (\mathbf{A}\mathbf{x} + \mathbf{B}\mathbf{u})dt + \sum_{i = 1}^{N}(\mathbf{C}_i\mathbf{x} + \mathbf{D}_i\mathbf{u})dW_i, \quad \mathbf{x}(0) = {\Xi}.\\
 \end{aligned}
\end{equation}
\begin{equation}\label{Eq6}
 \begin{aligned}
 &\mathcal{J}^{(N)}_{soc}(u,\xi_0) = \frac{1}{2}\mathbb{E}\int_{0}^{T} 
 \bigg\{\mathbf{x}^T\mathbf{Q}\mathbf{x} + 2\mathbf{S}_1^T\mathbf{x} + N\eta^TQ\eta + \mathbf{u}^T\mathbf{R}\mathbf{u} dt\\ &\hspace{4cm}+\mathbf{x}(T)^T\mathbf{G}\mathbf{x}(T) + 2\mathbf{S}_2^T\mathbf{x}(T) + N\bar{\eta}^TG\bar{\eta}\bigg\}, \\
 \end{aligned}
\end{equation}
where
\begin{equation}\label{Eq5}\resizebox{\textwidth}{!}{$
 \begin{aligned}
 & \mathbf{A}=
 \left(\begin{smallmatrix}
 A + \frac{F}{N} & \frac{F}{N} & \cdots & \frac{F}{N} \\
 \frac{F}{N} & A + \frac{F}{N} & \cdots & \frac{F}{N} \\
 \vdots & \vdots & \ddots &\vdots\\
 \frac{F}{N} & \frac{F}{N} & \cdots & A + \frac{F}{N} \\
 \end{smallmatrix}\right)_{(Nn\times Nn)},
 \mathbf{B}=
 \left(\begin{smallmatrix}
 B & 0 & \cdots & 0 \\
 0 & B & \cdots & 0 \\
 \vdots & \vdots & \ddots &\vdots\\
 0 & 0 & \cdots & B \\
 \end{smallmatrix}\right)_{(Nn\times Nm)},\ \ \Xi = \left(\begin{smallmatrix}
 \xi_0 \\
 \vdots\\
 \xi_0
 \end{smallmatrix}\right)_{(Nn\times 1)}, \ \ \mathbf{x}=
 \left(\begin{smallmatrix}
 x_1 \\
 \vdots\\
 x_N
 \end{smallmatrix}\right)_{(Nn\times 1)},\\
 & \mathbf{C}_i=\begin{smallmatrix}
 1 \\
 \vdots \\
 i^{\text{th}} \\
 \vdots \\
 N
 \end{smallmatrix}
 \left(\begin{smallmatrix}
 0 & \cdots & 0 & 0 & 0 & \cdots & 0\\
 \vdots & \ddots & \vdots & \vdots & \vdots & \ddots & \vdots\\
 \frac{\tilde{F}}{N} & \cdots &\frac{\tilde{F}}{N} & \frac{\tilde{F}}{N} + C &\frac{\tilde{F}}{N} &\cdots & \frac{\tilde{F}}{N}\\
 \vdots & \ddots & \vdots & \vdots & \vdots & \ddots & \vdots\\
 0 & \cdots & 0 & 0 & 0 & \cdots & 0\\
 \end{smallmatrix}\right)_{(Nn\times Nn)},\ \mathbf{D}_i=\begin{smallmatrix}
 1 \\
 \vdots \\
 i^{\text{th}} \\
 \vdots \\
 N
 \end{smallmatrix}\left(\begin{smallmatrix}
 0 & \cdots &  0  & \cdots & 0\\
 \vdots & \ddots  & \vdots  & \ddots & \vdots\\
 0 & \cdots  & D  & \cdots & 0\\
 \vdots & \ddots  & \vdots  & \ddots & \vdots\\
 0 & \cdots & 0  & \cdots & 0\\
 \end{smallmatrix}\right)_{(Nn\times Nm)},\ \ \mathbf{u}=
 \left(\begin{smallmatrix}
 u_1 \\
 \vdots\\
 u_N
 \end{smallmatrix}\right)_{(Nm\times 1)},
 \end{aligned} $}
\end{equation}
and
\begin{equation}\label{Eq7}\resizebox{\textwidth}{!}{$
 \begin{aligned}
 \mathbf{Q}=&
 \left(\begin{smallmatrix}
 Q + \frac{1}{N}(\Gamma^TQ\Gamma - Q\Gamma - \Gamma^TQ) & \frac{1}{N}(\Gamma^TQ\Gamma - Q\Gamma - \Gamma^TQ) & \cdots & \frac{1}{N}(\Gamma^TQ\Gamma - Q\Gamma - \Gamma^TQ) \\
 \frac{1}{N}(\Gamma^TQ\Gamma - Q\Gamma - \Gamma^TQ) & Q + \frac{1}{N}(\Gamma^TQ\Gamma - Q\Gamma - \Gamma^TQ) & \cdots & \frac{1}{N}(\Gamma^TQ\Gamma - Q\Gamma - \Gamma^TQ) \\
 \vdots & \vdots & \ddots &\vdots\\
 \frac{1}{N}(\Gamma^TQ\Gamma - Q\Gamma - \Gamma^TQ) & \frac{1}{N}(\Gamma^TQ\Gamma - Q\Gamma - \Gamma^TQ) & \cdots & Q + \frac{1}{N}(\Gamma^TQ\Gamma - Q\Gamma - \Gamma^TQ) \\
 \end{smallmatrix}\right)_{(Nn\times Nn)}
 =\left(\begin{smallmatrix}
 Q & 0 & \cdots & 0 \\
 0 & Q & \cdots & 0 \\
 \vdots & \vdots & \ddots &\vdots\\
 0 & 0 & \cdots & Q \\
 \end{smallmatrix}\right) + \frac{1}{N}\left(\begin{smallmatrix}
 \hat{Q}  & \cdots & \hat{Q} \\
 \vdots & \ddots &\vdots\\
  \hat{Q} & \cdots & \hat{Q} \\
 \end{smallmatrix}\right) - \frac{1}{N}\left(\begin{smallmatrix}
 {Q}  & \cdots & {Q} \\
 \vdots & \ddots &\vdots\\
  {Q} & \cdots & {Q} \\
 \end{smallmatrix}\right),\\
 \mathbf{G}=&\left(\begin{smallmatrix}
 G & 0 & \cdots & 0 \\
 0 & G & \cdots & 0 \\
 \vdots & \vdots & \ddots &\vdots\\
 0 & 0 & \cdots & G \\
 \end{smallmatrix}\right) + \frac{1}{N}\left(\begin{smallmatrix}
 \hat{G}  & \cdots & \hat{G} \\
 \vdots & \ddots &\vdots\\
  \hat{G} & \cdots & \hat{G} \\
 \end{smallmatrix}\right) - \frac{1}{N}\left(\begin{smallmatrix}
 {G}  & \cdots & {G} \\
 \vdots & \ddots &\vdots\\
  {G} & \cdots & {G} \\
 \end{smallmatrix}\right),\mathbf{S}_1 = -\left(\begin{smallmatrix}
 \Gamma^TQ\eta - Q\eta\\
 \vdots\\
 \Gamma^TQ\eta - Q\eta\\
 \end{smallmatrix}\right)_{(Nn\times 1)},
 \mathbf{S}_2 = -\left(\begin{smallmatrix}
 \bar{\Gamma}^TG\bar{\eta} - G\bar{\eta}\\
 \vdots\\
 \bar{\Gamma}^TG\bar{\eta} - G\bar{\eta}\\
 \end{smallmatrix}\right)_{(Nn\times 1)}, \mathbf{R} = \left(\begin{smallmatrix}
 R & 0 & \cdots & 0\\
 0 & R & \cdots & 0\\
 \vdots & \vdots & \ddots & \vdots\\
 0 & 0 & \cdots & R\\
 \end{smallmatrix}\right)_{(Nn\times Nn)},
 \end{aligned} $}
\end{equation}
where $\hat Q \triangleq (\Gamma - I)^TQ(\Gamma - I)$, $\hat G \triangleq (\bar{\Gamma} - I)^TG(\bar{\Gamma} - I)$.


Through \eqref{Eq4} and \eqref{Eq6}, we know that Problem \ref{P1} is actually a $nN$-dimensional standard control problem. Using the method in \cite{sly2016}, the solvability of Problem \ref{P1} and the optimal control can be derived. However,  the population $N$ is {usually large in realistic application, which would bring great computational complexity.} Thus, in what follows, some low-dimensional criteria for the convexity of the cost functional will be studied.
\subsection{Case 1: For Q, R, G are positive semi-definite }
We start with the simplest case that the weighting coefficients are all positive semi-definite. For the convexity of the cost functional, it follows that
\begin{proposition}
 Under {\rm\textbf{(A\ref{A1})}-\textbf{(A\ref{A2})}} and $Q, R,  G\geq0$, $\mathcal{J}^{(N)}_{soc}(u,\xi_0)$ is convex w.r.t (with respect to) $u$. Moreover, if $R\gg0$, then $\mathcal{J}^{(N)}_{soc}(u,\xi_0)$ is uniformly convex.
\end{proposition}
\begin{proof}
 Under the assumption $Q, R, G\geq0$, we know
$\left(\begin{smallmatrix}
 \hat{Q}  & \cdots & \hat{Q} \\
 \vdots & \ddots &\vdots\\
  \hat{Q} & \cdots & \hat{Q} \\
 \end{smallmatrix}\right)\geq 0$, $\left(\begin{smallmatrix}
 \hat{G}  & \cdots & \hat{G} \\
 \vdots & \ddots &\vdots\\
  \hat{G} & \cdots & \hat{G} \\
 \end{smallmatrix}\right)\geq 0$ and $\mathbf{R}\geq 0$.
 By the definition of positive semi-definiteness, we can obtain the following two inequalities
 \begin{equation*}
 \begin{aligned}
 &\left(\begin{smallmatrix}
 Q & 0 & \cdots & 0 \\
 0 & Q & \cdots & 0 \\
 \vdots & \vdots & \ddots &\vdots\\
 0 & 0 & \cdots & Q \\
 \end{smallmatrix}\right) - \frac{1}{N}\left(\begin{smallmatrix}
 {Q}  & \cdots & {Q} \\
 \vdots & \ddots &\vdots\\
  {Q} & \cdots & {Q} \\
 \end{smallmatrix}\right)\geq 0, &\left(\begin{smallmatrix}
 G & 0 & \cdots & 0 \\
 0 & G & \cdots & 0 \\
 \vdots & \vdots & \ddots &\vdots\\
 0 & 0 & \cdots & G \\
 \end{smallmatrix}\right) - \frac{1}{N}\left(\begin{smallmatrix}
 {G}  & \cdots & {G} \\
 \vdots & \ddots &\vdots\\
  {G} & \cdots & {G} \\
 \end{smallmatrix}\right)\geq 0.
 \end{aligned}
\end{equation*}
Further, the convexity would follow. Moreover, if $R\gg 0$, then $\mathbf{R}\gg 0$ and $\mathcal{J}^{(N)}_{soc}(u,\xi_0)$ is uniformly convex.
\end{proof}
Next, we consider a more general situation.
\subsection{Case 2: For Q, R and G are all indefinite}
In this case, there are two situations need to be discussed: $F = \tilde{F} = 0$ and $F,\tilde{F}\neq 0$.
Now, we consider the first situation: $F = \tilde{F} = 0$.
\subsubsection{For $F = \tilde{F} = 0$}
{In this situation}, the agent's state will not be influenced by others (see \cite{hcm2012}).   In this situation, the {agents' state dynamics}  are decoupled. For the cost functional, we {have} the following proposition.
\begin{proposition}\label{prop2}
 Suppose that  $Q-\hat{Q}\geq 0$ and $G-\hat{G}\geq 0$, then for any $\Delta Q$, $\Delta G\in\mathbb{S}^n$ such that $\Delta Q \geq Q - \hat{Q}$ and $\Delta G \geq G - \hat{G}$, we have $\mathbf{Q} - \mathbf{Q}_2 \geq 0$ and $\mathbf{G} - \mathbf{G}_2 \geq 0$ where
 \begin{equation*}
 \begin{aligned}
 &\mathbf{Q}_2 = \left(\begin{smallmatrix}
 Q - \Delta Q & 0 & \cdots & 0 \\
 0 & Q - \Delta Q & \cdots & 0 \\
 \vdots & \vdots & \ddots &\vdots\\
 0 & 0 & \cdots & Q - \Delta Q \\
 \end{smallmatrix}\right),\ \ \mathbf{G}_2 = \left(\begin{smallmatrix}
 G - \Delta G & 0 & \cdots & 0 \\
 0 & G - \Delta G & \cdots & 0 \\
 \vdots & \vdots & \ddots &\vdots\\
 0 & 0 & \cdots & G - \Delta G \\
 \end{smallmatrix}\right).
 \end{aligned}
\end{equation*}
\end{proposition}
\begin{proof}
 Consider matrix $\mathbf{Q}_2$, $\mathbf{G}_2$, $\mathbf{Q} - \mathbf{Q}_2$ and $\mathbf{G} - \mathbf{G}_2$ which are
\begin{equation*}
\resizebox{\textwidth}{!}{$
 \begin{aligned}
 &\mathbf{Q}_2 \triangleq \left(\begin{smallmatrix}
 Q - \Delta Q & 0 & \cdots & 0 \\
 0 & Q - \Delta Q & \cdots & 0 \\
 \vdots & \vdots & \ddots &\vdots\\
 0 & 0 & \cdots & Q - \Delta Q \\
 \end{smallmatrix}\right),\ \
 \mathbf{Q} - \mathbf{Q}_2 = \left(\begin{smallmatrix}
 \Delta Q + \frac{1}{N}\hat{Q} - \frac{1}{N}Q & \frac{1}{N}\hat{Q} - \frac{1}{N}Q & \cdots & \frac{1}{N}\hat{Q} - \frac{1}{N}Q \\
 \frac{1}{N}\hat{Q} - \frac{1}{N}Q& \Delta Q +\frac{1}{N}\hat{Q} - \frac{1}{N}Q & \cdots & \frac{1}{N}\hat{Q} - \frac{1}{N}Q \\
 \vdots & \vdots & \ddots &\vdots\\
 \frac{1}{N}\hat{Q} - \frac{1}{N}Q & \frac{1}{N}\hat{Q} - \frac{1}{N}Q& \cdots & \Delta Q +\frac{1}{N}\hat{Q} - \frac{1}{N}Q\\
 \end{smallmatrix}\right),\\
 &\mathbf{G}_2 \triangleq \left(\begin{smallmatrix}
 G - \Delta G & 0 & \cdots & 0 \\
 0 & G - \Delta G & \cdots & 0 \\
 \vdots & \vdots & \ddots &\vdots\\
 0 & 0 & \cdots & G - \Delta G \\
 \end{smallmatrix}\right),\ \
 \mathbf{G} - \mathbf{G}_2 = \left(\begin{smallmatrix}
 \Delta G + \frac{1}{N}\hat{G} - \frac{1}{N}G & \frac{1}{N}\hat{G} - \frac{1}{N}G & \cdots & \frac{1}{N}\hat{G} - \frac{1}{N}G \\
 \frac{1}{N}\hat{G} - \frac{1}{N}G& \Delta G +\frac{1}{N}\hat{G} - \frac{1}{N}G & \cdots & \frac{1}{N}\hat{G} - \frac{1}{N}G \\
 \vdots & \vdots & \ddots &\vdots\\
 \frac{1}{N}\hat{G} - \frac{1}{N}G & \frac{1}{N}\hat{G} - \frac{1}{N}G& \cdots & \Delta G +\frac{1}{N}\hat{G} - \frac{1}{N}G\\
 \end{smallmatrix}\right).
 \end{aligned}$}
\end{equation*}
If $Q-\hat{Q}\geq 0$ holds, then for any non-zeros vector $(x_1^T,\cdots,x_N^T)^T\in\mathbb{R}^{Nn\times 1}$ and any $\Delta Q \geq Q - \hat{Q}$,
\begin{equation*}
 \begin{aligned}
 &\left(\begin{smallmatrix}
 x_1 \\
 x_2 \\
 \vdots \\
 x_N
 \end{smallmatrix}\right)^T \left(\begin{smallmatrix}
 \Delta Q + \frac{1}{N}\hat{Q} - \frac{1}{N}Q & \frac{1}{N}\hat{Q} - \frac{1}{N}Q & \cdots & \frac{1}{N}\hat{Q} - \frac{1}{N}Q \\
 \frac{1}{N}\hat{Q} - \frac{1}{N}Q& \Delta Q +\frac{1}{N}\hat{Q} - \frac{1}{N}Q & \cdots & \frac{1}{N}\hat{Q} - \frac{1}{N}Q \\
 \vdots & \vdots & \ddots &\vdots\\
 \frac{1}{N}\hat{Q} - \frac{1}{N}Q & \frac{1}{N}\hat{Q} - \frac{1}{N}Q& \cdots & \Delta Q +\frac{1}{N}\hat{Q} - \frac{1}{N}Q\\
 \end{smallmatrix}\right)\left(\begin{smallmatrix}
 x_1 \\
 x_2 \\
 \vdots \\
 x_N
 \end{smallmatrix}\right)\\
 =&(x_1^T\Delta Q x_1 + \cdots + x_N^T\Delta Q x_N) + \frac{1}{N}(x_1 + \cdots + x_N)^T\hat{Q}(x_1 + \cdots + x_N)\\
 &- \frac{1}{N}(x_1 + \cdots + x_N)^T{Q}(x_1 + \cdots + x_N)\\
 =&(x_1^T\Delta Q x_1 + \cdots + x_N^T\Delta Q x_N) - \frac{1}{N}(x_1 + \cdots + x_N)^T({Q} - \hat{Q})(x_1 + \cdots + x_N)\\
 \geq&(x_1^T\Delta Q x_1 + \cdots + x_N^T\Delta Q x_N) - (x_1^T({Q}-\hat{Q})x_1 + \cdots + x_N^T({Q}-\hat{Q})x_N)\\
 =&x_1^T(\Delta Q - ({Q}-\hat{Q}))x_1 + \cdots + x_N^T(\Delta Q - ({Q}-\hat{Q}))x_N \geq 0.
 \end{aligned}
\end{equation*}
By the same argument, one can also obtain
\begin{equation*}
 \begin{aligned}
 &\left(\begin{smallmatrix}
 x_1 \\
 x_2 \\
 \vdots \\
 x_N
 \end{smallmatrix}\right)^T \left(\begin{smallmatrix}
 \Delta G + \frac{1}{N}\hat{G} - \frac{1}{N}G & \frac{1}{N}\hat{G} - \frac{1}{N}G & \cdots & \frac{1}{N}\hat{G} - \frac{1}{N}G \\
 \frac{1}{N}\hat{G} - \frac{1}{N}G& \Delta G +\frac{1}{N}\hat{G} - \frac{1}{N}G & \cdots & \frac{1}{N}\hat{G} - \frac{1}{N}G \\
 \vdots & \vdots & \ddots &\vdots\\
 \frac{1}{N}\hat{G} - \frac{1}{N}G & \frac{1}{N}\hat{G} - \frac{1}{N}G& \cdots & \Delta G +\frac{1}{N}\hat{G} - \frac{1}{N}G\\
 \end{smallmatrix}\right)\left(\begin{smallmatrix}
 x_1 \\
 x_2 \\
 \vdots \\
 x_N
 \end{smallmatrix}\right)\geq 0.
 \end{aligned}
\end{equation*}
Thus, $\mathbf{Q} - \mathbf{Q}_2 \geq 0$ and $\mathbf{G} - \mathbf{G}_2 \geq 0$ and the proposition is proved.
\end{proof}
Consequently, Proposition \ref{prop2} implies
\begin{equation}\label{Eq8}
\resizebox{\textwidth}{!}{$
 \begin{aligned}
 &\frac{1}{2}\mathbb{E}\bigg\{\int_{0}^{T} \mathbf{x}^T\mathbf{Q}\mathbf{x} + \mathbf{u}^T\mathbf{R}\mathbf{u} dt +\mathbf{x}(T)^T\mathbf{G}\mathbf{x}(T)\bigg\} \geq \frac{1}{2}\mathbb{E}\bigg\{\int_{0}^{T} \mathbf{x}^T\mathbf{Q}_2\mathbf{x} + \mathbf{u}^T\mathbf{R}\mathbf{u} dt+\mathbf{x}(T)^T\mathbf{G}_2\mathbf{x}(T)\bigg\}.
 \end{aligned}$}
\end{equation}
Motivated by \eqref{Eq8},
we have the following result:
\begin{proposition}\label{prop3.3}
 Under \text{\rm\textbf{(A1)}-\textbf{(A2)}}, $F = \tilde{F} = 0$, $Q-\hat{Q} \geq 0$ and $G-\hat{G} \geq 0$, if there exist some $\Delta Q$, $\Delta G\in\mathbb{S}^n$ such that $\Delta Q \geq Q - \hat{Q}$, $\Delta G \geq G - \hat{G}$ and the following low-dimensional control problem
\begin{equation}\label{Eq9}
\left\{
 \begin{aligned}
 &\min \frac{1}{2}\mathbb{E}\bigg\{\int_{0}^{T} {x}^T(Q - \Delta Q){x} + {u}^T{R}{u} dt+{x}^T(T)(G - \Delta G){x}(T)\bigg\},\\
 &s.t.\quad 
 \begin{aligned}
 & d{x} = ({A}{x} + {B}{u})dt + ({C}{x} + {D}{u})dW, \quad {x}(0) = 0.\\
 \end{aligned}
 \end{aligned}
 \right.
\end{equation} is (uniformly) convex,  then $\mathcal{J}^{(N)}_{soc}(u)$ is (uniformly) convex.
\end{proposition}
\begin{proof}
 According to the result of \cite{sly2016}, $\mathcal{J}^{(N)}_{soc}(u,\xi_0)$ is (uniformly) convex if $\widetilde{\mathcal{J}}^{(N)}_{soc}(u,0)\geq 0$ ($\geq \varepsilon\mathbb{E}\int_{0}^{T}\|\mathbf{u}\|^2dt$ for some constant $\varepsilon$) where
\begin{equation*}
\left\{
 \begin{aligned}
 &\widetilde{\mathcal{J}}^{(N)}_{soc}(u,0) \triangleq \frac{1}{2}\mathbb{E}\bigg\{\int_{0}^{T} \mathbf{x}^T\mathbf{Q}\mathbf{x} + \mathbf{u}^T\mathbf{R}\mathbf{u} dt +\mathbf{x}(T)^T\mathbf{G}\mathbf{x}(T)\bigg\}, \\
 &s.t.\quad
 \begin{aligned}
 & d\mathbf{x} = (\mathbf{A}\mathbf{x} + \mathbf{B}\mathbf{u})dt + \sum_{i = 1}^{N}(\mathbf{C}_i\mathbf{x} + \mathbf{D}_i\mathbf{u})dW_i, \quad \mathbf{x}(0) = 0.\\
 \end{aligned}
 \end{aligned}
 \right.
\end{equation*}
By (A\ref{A1})-(A\ref{A2}) and  \eqref{Eq8}, for any $\Delta Q \geq Q - \hat{Q}$ and $\Delta G \geq G - \hat{G}$, if
$\frac{1}{2}\mathbb{E}\big\{\int_{0}^{T} \mathbf{x}^T\mathbf{Q}_2\mathbf{x} + \mathbf{u}^T\mathbf{R}\mathbf{u} dt +\mathbf{x}(T)^T\mathbf{G}_2\mathbf{x}(T)\big\} \geq 0$ ($\geq \varepsilon\mathbb{E}\int_{0}^{T}\|\mathbf{u}\|^2dt$) holds, then the (uniform) convexity of $\mathcal{J}^{(N)}_{soc}(u,\xi_0)$ would follow. Moreover, $\mathbf{Q}_2$ is diagonal block matrix and the following problem is decouple.
\begin{equation}\label{Eq10}
\left\{
 \begin{aligned}
 &\min \frac{1}{2}\mathbb{E} \bigg\{\int_{0}^{T} \mathbf{x}^T\mathbf{Q}_2\mathbf{x} + \mathbf{u}^T\mathbf{R}\mathbf{u} dt +\mathbf{x}(T)^T\mathbf{G}_2\mathbf{x}(T)\bigg\},\\
 &s.t.\quad 
 \begin{aligned}
 & d\mathbf{x} = (\mathbf{A}\mathbf{x} + \mathbf{B}\mathbf{u})dt + \sum_{i = 1}^{N}(\mathbf{C}_i\mathbf{x} + \mathbf{D}_i\mathbf{u})dW_i, \quad \mathbf{x}(0) = 0.\\
 \end{aligned}
 \end{aligned}
 \right.
\end{equation}
Thus, the convexity of \eqref{Eq10} is equivalent to the convexity of \eqref{Eq9}, and $\frac{1}{2}\mathbb{E}\big\{\int_{0}^{T} \mathbf{x}^T\mathbf{Q}_2\mathbf{x} + \mathbf{u}^T\mathbf{R}\mathbf{u} dt +\mathbf{x}(T)^T\mathbf{G}_2\mathbf{x}(T)\big\} \geq 0$ ($\geq \varepsilon\mathbb{E}\int_{0}^{T}\|\mathbf{u}\|^2dt$) if \eqref{Eq9} is (uniformly) convex. Theorem \ref{prop3.3} is proved.
\end{proof}

 Note that the convexity of a low-dimensional control problem has been well studied in \cite{sly2016} and \cite{sy2014}. Thus, we will not further discuss the convexity of \eqref{Eq9} here.

Next, the second situation: $F,\tilde{F}\neq 0$ will be studied. {In this situation, for discussion simplicity, we also assume that $G\geq 0$}.
\subsubsection{For $G\geq0$, $F,  \tilde{F} \neq 0$}
In this situation, { the state of each agents will be influenced by the others (see \cite{wang2019social} where $F\neq 0$ is assumed).} It is inaccessible to obtain a decoupled low-dimensional problem due to the coupling of the state dynamics. To analyse $\widetilde{\mathcal{J}}^{(N)}_{soc}(u,0)$, we firstly estimate the $L^2$ norm of $\mathbf{x}$ which is given by \eqref{Eq10}. Applying It\^{o}'s formula to $\mathbf{x}$ and taking expectation, one can obtain
\begin{equation}\label{Eq11}
\resizebox{\textwidth}{!}{$
 \begin{aligned}
 & \mathbb{E}\|\mathbf{x}(t)\|^2 = \mathbb{E}\int_{0}^{t}\Big[\mathbf{x}^T\Big(\mathbf{A}+\mathbf{A}^T+\sum_{i = 1}^{N}\mathbf{C}_i^T\mathbf{C}_i\Big)\mathbf{x} + \mathbf{x}^T\Big(2\mathbf{B}+2\sum_{i = 1}^{N}\mathbf{C}_i^T\mathbf{D}_i\Big)\mathbf{u} + \mathbf{u}^T\sum_{i = 1}^{N}\mathbf{D}_i^T\mathbf{D}_i\mathbf{u}\Big]ds,\\
 \end{aligned}$}
\end{equation}
where
\begin{equation*}
\resizebox{\textwidth}{!}{$
 \begin{aligned}
 &\sum_{i = 1}^{N}\mathbf{C}_i^T\mathbf{C}_i = \left(\begin{smallmatrix}
 C^TC & 0 & \cdots & 0 \\
 0 & C^TC & \cdots & 0 \\
 \vdots & \vdots & \ddots &\vdots\\
 0 & 0 & \cdots & C^TC \\
 \end{smallmatrix}\right) + \frac{1}{N}\left(\begin{smallmatrix}
 (\tilde{F}+C)^T(\tilde{F}+C) & \cdots & (\tilde{F}+C)^T(\tilde{F}+C) \\
 \vdots & \ddots &\vdots\\
 (\tilde{F}+C)^T(\tilde{F}+C) & \cdots & (\tilde{F}+C)^T(\tilde{F}+C) \\
 \end{smallmatrix}\right) - \frac{1}{N}\left(\begin{smallmatrix}
 C^TC & \cdots & C^TC \\
 \vdots & \ddots &\vdots\\
 C^TC & \cdots & C^TC \\
 \end{smallmatrix}\right)\geq 0,\\
 &\mathbf{A}^T + \mathbf{A} = \left(\begin{smallmatrix}
 A^T + A & 0 & \cdots & 0 \\
 0 & A^T + A & \cdots & 0 \\
 \vdots & \vdots & \ddots &\vdots\\
 0 & 0 & \cdots & A^T + A \\
 \end{smallmatrix}\right) + \frac{1}{N}\left(\begin{smallmatrix}
 F^T+F & \cdots & F^T+F \\
 \vdots & \ddots &\vdots\\
 F^T+F & \cdots & F^T+F \\
 \end{smallmatrix}\right)\\
 &\hspace{1.4cm}\leq\left(\begin{smallmatrix}
 A^T + A & 0 & \cdots & 0 \\
 0 & A^T + A & \cdots & 0 \\
 \vdots & \vdots & \ddots &\vdots\\
 0 & 0 & \cdots & A^T + A \\
 \end{smallmatrix}\right) +\left(\begin{smallmatrix}
 \lambda_{\max}(F^T+F)I & \cdots & 0 \\
 \vdots & \ddots &\vdots\\
 0 & \cdots & \lambda_{\max}(F^T+F)I \\
 \end{smallmatrix}\right),\\
 &\sum_{i = 1}^{N}\mathbf{D}_i^T\mathbf{D}_i = \left(\begin{smallmatrix}
 D^TD & 0 & \cdots & 0 \\
 0 & D^TD & \cdots & 0 \\
 \vdots & \vdots & \ddots &\vdots\\
 0 & 0 & \cdots & D^TD \\
 \end{smallmatrix}\right), \ \
 \sum_{i = 1}^{N}\mathbf{C}_i^T\mathbf{D}_i = \frac{1}{N}\left(\begin{smallmatrix}
 \tilde{F}^TD & \cdots & \tilde{F}^TD  \\
 \vdots & \ddots & \vdots  \\
 \tilde{F}^TD & \cdots & \tilde{F}^TD  \\
 \end{smallmatrix}\right) + \left(\begin{smallmatrix}
 C^TD & 0 & \cdots & 0 \\
 0 & C^TD & \cdots & 0 \\
 \vdots & \vdots & \ddots & \vdots \\
 0 & 0 & \cdots & C^TD \\
 \end{smallmatrix}\right),\\
  \end{aligned}$}
\end{equation*}
\begin{equation*}\resizebox{\textwidth}{!}{$
\begin{aligned}
 &\sum_{i = 1}^{N}\mathbf{D}_i^T\mathbf{C}_i\sum_{i = 1}^{N}\mathbf{C}_i^T\mathbf{D}_i = \frac{1}{N}\left(\begin{smallmatrix}
 D^T \tilde{F} & \cdots &D^T \tilde{F} \\
 \vdots & \ddots & \vdots \\
 D^T \tilde{F} & \cdots & D^T \tilde{F} \\
 \end{smallmatrix}\right)\frac{1}{N}\left(\begin{smallmatrix}
 \tilde{F}^TD & \cdots & \tilde{F}^TD \\
 \vdots & \ddots & \vdots  \\
 \tilde{F}^TD & \cdots & \tilde{F}^TD  \\
 \end{smallmatrix}\right)
 + \left(\begin{smallmatrix}
 D^TC & 0 & \cdots & 0 \\
 0 & D^TC & \cdots & 0 \\
 \vdots & \vdots & \ddots & \vdots \\
 0 & 0 & \cdots & D^TC \\
 \end{smallmatrix}\right)\frac{1}{N}\left(\begin{smallmatrix}
 \tilde{F}^TD & \cdots & \tilde{F}^TD \\
 \vdots & \ddots & \vdots \\
 \tilde{F}^TD & \cdots & \tilde{F}^TD \\
 \end{smallmatrix}\right)\\
 &\hspace{0cm} + \frac{1}{N}\left(\begin{smallmatrix}
 D^T \tilde{F} & \cdots &D^T \tilde{F}  \\
 \vdots & \ddots & \vdots  \\
 D^T \tilde{F} & \cdots & D^T \tilde{F}  \\
 \end{smallmatrix}\right)\left(\begin{smallmatrix}
 C^TD & 0 & \cdots & 0 \\
 0 & C^TD & \cdots & 0 \\
 \vdots & \vdots & \ddots & \vdots \\
 0 & 0 & \cdots & C^TD \\
 \end{smallmatrix}\right)+ \left(\begin{smallmatrix}
 D^TC & 0 & \cdots & 0 \\
 0 & D^TC & \cdots & 0 \\
 \vdots & \vdots & \ddots & \vdots \\
 0 & 0 & \cdots & D^TC \\
 \end{smallmatrix}\right)\left(\begin{smallmatrix}
 C^TD & 0 & \cdots & 0 \\
 0 & C^TD & \cdots & 0 \\
 \vdots & \vdots & \ddots & \vdots \\
 0 & 0 & \cdots & C^TD \\
 \end{smallmatrix}\right)\\
 &\hspace{0cm}= \left(\begin{smallmatrix}
 D^TCC^TD & 0 & \cdots & 0 \\
 0 & D^TCC^TD & \cdots & 0 \\
 \vdots & \vdots & \ddots & \vdots \\
 0 & 0 & \cdots & D^TCC^TD \\
 \end{smallmatrix}\right)+ \frac{1}{N}\left(\begin{smallmatrix}
 D^T (\tilde{F}\tilde{F}^T + C\tilde{F}^T + \tilde{F}C^T)D & \cdots & D^T (\tilde{F}\tilde{F}^T + C\tilde{F}^T + \tilde{F}C^T)D \\
 \vdots & \ddots & \vdots \\
 D^T (\tilde{F}\tilde{F}^T + C\tilde{F}^T + \tilde{F}C^T)D & \cdots & D^T (\tilde{F}\tilde{F}^T + C\tilde{F}^T + \tilde{F}C^T)D \\
 \end{smallmatrix}\right).
 \end{aligned} $}
\end{equation*}
Moreover, to estimate $\mathbf{x}$, we still need the following result:
\begin{proposition}\label{prop3.4}
 Suppose $S$ is a real symmetric matrix, then for any real vector $x$,
 \[\lambda_{\min}(S)\|x\|^2 \leq x^TSx \leq \lambda_{\max}(S)\|x\|^2.\]
\end{proposition}
\begin{proof}
 $S\in\mathbb{S}^{n}$ is a real symmetric matrix. There exists an orthogonal matrix $P$ such that $P^{-1}SP = \Lambda$, where $\Lambda$ is diagonal and the diagonal elements are the eigenvalues $\lambda _{1}$, $\dots$, $\lambda_{n}$ of $S$. Hence, by letting $P^{T}x = y = (y_1^T,\cdots,y_n^T)^T$, we have
 \[x^TSx = x^TP\Lambda P^{T}x=\lambda_1y_1^2+\cdots+\lambda_ny_n^2.\]
 This implies that $\lambda_{\min}(S)\|y\|^2 \leq x^TSx \leq \lambda_{\max}(S)\|y\|^2$. Noting the orthogonality of $P$, we have $\|y\|^2 = x^TPP^Tx = \|x\|^2$ and the result of Proposition \ref{prop3.4} follows.
\end{proof}

Based on \eqref{Eq11} and Proposition \ref{prop3.4}, the estimation of $\mathbb{E}\|\mathbf{x}(t)\|^2$ and the convexity of $\mathcal{J}^{(N)}_{soc}(u,\xi_0)$ can be obtained as follows.
\begin{proposition}\label{prop3.5}
 Under \text{\rm\textbf{({A1})}-\textbf{({A2})}}, $G\geq0$, $F,  \tilde{F} \neq 0$ and $Q-\hat{Q} \geq 0$, if there exist some $\Delta Q\in\mathbb{S}^n$  such that $\Delta Q \geq Q - \hat{Q}$, $\lambda_{\min}(Q - \Delta Q)\leq 0$ and $Ke^{2KT}\lambda_{\min}(Q - \Delta Q) + \frac{1}{2}\lambda_{\min}(R) \geq0$ (or $\geq \varepsilon I$ for some $\varepsilon>0$), then $\mathcal{J}^{(N)}_{soc}(u,\xi_0)$ is convex (or uniformly convex). $K$ is given by
\begin{equation*}
\resizebox{\textwidth}{!}{$
 \begin{aligned}
 K = \max\Big\{&\big(\lambda_{\max}(A^T + A) + \lambda_{\max}(F^T+F)\big),\lambda_{\max}\big(C^TC + (\tilde{F}+C)^T(\tilde{F}+C)\big),\sqrt{\lambda_{\max}(B^TB)},\\
 &\sqrt{\lambda_{\max}\big(D^T (\tilde{F}\tilde{F}^T + C\tilde{F}^T + \tilde{F}C^T)D\big) + \lambda_{\max}\big(D^TCC^TD\big)},\lambda_{\max}(D^TD)\Big\}.
 \end{aligned}$}
\end{equation*}
Moreover, if $Q\leq 0$ is assumed, then condition $\lambda_{\min}(Q - \Delta Q)\leq 0$ would follow.
\end{proposition}
\begin{proof}
By Proposition \ref{prop3.4} and noting that the non-zero eigenvalues of $\frac{1}{N}\left(\begin{smallmatrix}
 A & \cdots & A \\
 \vdots & \vdots &\vdots\\
 A & \cdots & A \\
 \end{smallmatrix}\right)$, $\left(\begin{smallmatrix}
 A & & \mathbf{0} \\
 & \ddots &\\
 \mathbf{0} & & A \\
 \end{smallmatrix}\right)$ and $A$ are the same, we have
 \begin{equation*}
 \resizebox{\textwidth}{!}{$
 \left\{
 \begin{aligned}
 &\mathbf{x}^T\left(\sum_{i = 1}^{N}\mathbf{C}_i^T\mathbf{C}_i\right)\mathbf{x} \leq \lambda_{\max}\Big(C^TC + (\tilde{F}+C)^T(\tilde{F}+C)\Big)\|\mathbf{x}\|^2,\\
 &\mathbf{x}^T\left(\mathbf{A}^T + \mathbf{A}\right)\mathbf{x} \leq \big[\lambda_{\max}(A^T + A) + \lambda_{\max}(F^T+F)\big] \|\mathbf{x}\|^2,\\
 &2\mathbf{x}^T\mathbf{B}\mathbf{u} = \langle\mathbf{B}\mathbf{u},\mathbf{x}\rangle + \langle\mathbf{u},\mathbf{B}^T\mathbf{x}\rangle\leq \sqrt{\langle\mathbf{B}\mathbf{u},\mathbf{B}\mathbf{u}\rangle \langle\mathbf{x},\mathbf{x}\rangle} + \sqrt{\langle\mathbf{u},\mathbf{u}\rangle \langle\mathbf{B}^T\mathbf{x},\mathbf{B}^T\mathbf{x}\rangle}\\
 &\hspace{4.55cm}\leq 2\sqrt{\lambda_{\max}(B^TB)}\|\mathbf{u}\|\|\mathbf{x}\|,\\
 &2\mathbf{x}^T\sum_{i = 1}^{N}\mathbf{C}_i^T\mathbf{D}_i\mathbf{u} \leq 2\sqrt{\lambda_{\max}\bigg(\sum_{i = 1}^{N}\mathbf{D}_i^T\mathbf{C}_i\sum_{i = 1}^{N}\mathbf{C}_i^T\mathbf{D}_i\bigg)}\|\mathbf{u}\|\|\mathbf{x}\|\\
 &\hspace{2.7cm}\leq 2\sqrt{\lambda_{\max}\big(D^T (\tilde{F}\tilde{F}^T + C\tilde{F}^T + \tilde{F}C^T)D\big)+\lambda_{\max}\big(D^TCC^TD\big)}\|\mathbf{u}\|\|\mathbf{x}\|,\\
 &\mathbf{u}^T\sum_{i = 1}^{N}\mathbf{D}_i^T\mathbf{D}_i\mathbf{u} \leq \lambda_{\max}(D^TD)\|\mathbf{u}\|^2.
 \end{aligned}
 \right.$}
 \end{equation*}
 Thus, the estimation below would follows
\begin{equation*}
\resizebox{\textwidth}{!}{$
 \begin{aligned}
 & \mathbb{E}\|\mathbf{x}(t)\|^2 = \mathbb{E}\int_{0}^{t}\Big[\mathbf{x}^T\Big(\mathbf{A}+\mathbf{A}^T+\sum_{i = 1}^{N}\mathbf{C}_i^T\mathbf{C}_i\Big)\mathbf{x} + \mathbf{x}^T\Big(2\mathbf{B}+2\sum_{i = 1}^{N}\mathbf{C}_i^T\mathbf{D}_i\Big)\mathbf{u} + \mathbf{u}^T\sum_{i = 1}^{N}\mathbf{D}_i^T\mathbf{D}_i\mathbf{u}\Big]ds\\
 &\hspace{1.5cm}\leq K\mathbb{E}\int_{0}^{t}\Big[\|\mathbf{x}\|^2 + 2\|\mathbf{u}\|\|\mathbf{x}\| + \|\mathbf{u}\|^2\Big]ds\leq 2K\int_{0}^{t}\Big[\mathbb{E}\|\mathbf{x}\|^2 + \mathbb{E}\|\mathbf{u}\|^2\Big]ds,
 \end{aligned}$}
\end{equation*}
where
\begin{equation*}
\resizebox{\textwidth}{!}{$
 \begin{aligned}
 K = \max\Big\{&\big(\lambda_{\max}(A^T + A) + \lambda_{\max}(F^T+F)\big),\lambda_{\max}\big(C^TC + (\tilde{F}+C)^T(\tilde{F}+C)\big),\sqrt{\lambda_{\max}(B^TB)},\\
 &\sqrt{\lambda_{\max}\big(D^T (\tilde{F}\tilde{F}^T + C\tilde{F}^T + \tilde{F}C^T)D\big) + \lambda_{\max}\big(D^TCC^TD\big)},\lambda_{\max}(D^TD)\Big\}\geq 0.
 \end{aligned}
 $}
\end{equation*}
Further, by Gronwall's inequality, $\mathbb{E}\|\mathbf{x}(t)\|^2\leq 2Ke^{2Kt}\|\mathbf{u}\|_{L^2}^2$. Noting that $G \geq 0$ and $\lambda_{\min}(Q - \Delta Q)\leq 0$, the estimation of $\widetilde{\mathcal{J}}^{(N)}_{soc}(u,0)$ can be derived as
\begin{equation*}
 \begin{aligned}
 &\frac{1}{2}\mathbb{E}\bigg\{\int_{0}^{T} \mathbf{x}^T\mathbf{Q}\mathbf{x} + \mathbf{u}^T\mathbf{R}\mathbf{u} dt +\mathbf{x}(T)^T\mathbf{G}\mathbf{x}(T)\bigg\}\\
 \geq &\frac{1}{2}\mathbb{E}\bigg\{\int_{0}^{T} \mathbf{x}^T\mathbf{Q}_2\mathbf{x} + \mathbf{u}^T\mathbf{R}\mathbf{u} dt\bigg\} \geq   Ke^{2KT}\lambda_{\min}(Q - \Delta Q) \|\mathbf{u}\|_{L^2}^2 + \frac{1}{2}\lambda_{\min}(R)\|\mathbf{u}\|^2_{L^2}, \\
 \end{aligned}
\end{equation*}
and {Proposition} \ref{prop3.5} can be obtained.

Moreover, if we further assume that $Q\leq 0$, then $\hat{Q} = (\Gamma - I)^TQ(\Gamma - I)\leq 0$. Thus, it follows that $ 0\geq \hat{Q} \geq Q - \Delta Q$ and condition $\lambda_{\min}(Q - \Delta Q)\leq 0$ will follow as well.
\end{proof}
\begin{remark}
  $\mathbf{x}(t)$ is driven by $\mathbf{u}(t)$, and $\|\mathbf{x}\|^2_{L^2}$ can be bounded by the multiple of $\|\mathbf{u}\|^2_{L^2}$ via the Gronwall's inequality. Thus, $\widetilde{\mathcal{J}}^{(N)}_{soc}(u,0)$ can also be estimated by $\|\mathbf{u}\|^2_{L^2}$. Though the conditions $\lambda_{\min}(Q - \Delta Q)\leq 0$ (or $Q\leq 0$) and $Ke^{2KT}\lambda_{\min}(Q - \Delta Q) + \frac{1}{2}\lambda_{\min}(R) \geq0$ in Proposition \ref{prop3.5}, we see that $R\geq 0$ should hold to neutralize the negative definiteness of $Q$.  However, {the opposite is not true}. $\|\mathbf{u}\|^2_{L^2}$ can not be bounded by $\|\mathbf{x}\|^2_{L^2}$.
  Thus, if $R\leq 0$, we can not estimate how ``large" $Q$ should be to neutralize the negative definiteness of $R$.
\end{remark}
Through the discussion above, we have studied the (uniform) convexity of the social cost functional. For the sake of simplicity, we introduce the following uniform convexity assumption.
\begin{assumption}\label{A3}{\rm\textbf{(A3)}}
  $\mathcal{J}^{(N)}_{soc}(u)$ is uniformly convex w.r.t. $\mathbf{u}$.
\end{assumption}
Under \textbf{(A\ref{A1})}-\textbf{(A\ref{A3})}, $\mathcal{J}^{(N)}_{soc}(u)$ is uniformly convex, and consequently Problem \ref{P1}, \ref{P2} are both uniquely solvable. In what follows, the problem will be discussed under {such} assumption. {Next, we will apply variational method to analyse the social cost functional, and derive some decentralized controls based on person-by-person optimality principle}.

\section{Person-by-person Optimality}
{ Person-by-person optimality is a critical technique in mean-field social optima scheme, which has been applied in many recent social optima literatures (e.g., \cite{bjj2017}, \cite{wang2019social}, \cite{bj2017b}). {It is very different from mean-field games scheme, where the auxiliary control problem is usually derived directly by fixing the state-average, and this would lead to some ineffective control in social optima scheme}. Thus, in this section, under the person-by-person optimality principle, variation method will be used to analysis the mean-field approximation.}

Consider the optimal centralized control $\bar{u}=(\bar{u}_1,\bar{u}_2,\cdots,\bar{u}_N)$  of the large population system. $\bar{x}=(\bar{x}_1,\bar{x}_2,\cdots,\bar{x}_N)$ denotes the associated optimal state. {We perturb $\bar{u}_i$ and keep $\bar{u}_{ - i}\triangleq(\bar{u}_1 , \cdots , \bar{u}_{i - 1} , \bar{u}_{i + 1} , \cdots , \bar{u}_N)$ fixed. The perturbed control is denoted by $\bar{u}_i+\delta u_i$, while the associated perturbed state is denoted by $(\bar{x}_1+\delta x_1,\cdots,\bar{x}_N+\delta x_N)$}. The perturbation of the cost functional is denoted by $\delta \mathcal{J}_j\triangleq \mathcal{J}_j(u_j,\bar{u}_{ - i}) - \mathcal{J}_j(\bar{u}_j,\bar{u}_{ - i})$, for $j = 1,\cdots,N$. Thus, the dynamic  of $\mathcal{A}_i$ state variation is
\begin{equation*}
 \begin{aligned}
 & d\delta x_i = (A\delta x_i + B\delta {u}_i + F\delta x^{(N)})dt + (C\delta x_i + D\delta {u}_i + \tilde{F}\delta x^{(N)})dW_i, \quad \delta x_i(0) = 0, \\
 \end{aligned}
\end{equation*}
and the dynamics of  the others' variations  are
\begin{equation}\label{Eq12}
 \begin{aligned}
 & d\delta {x}_j = (A\delta {x}_j + F\delta x^{(N)})dt + (C\delta {x}_j + \tilde{F}\delta x^{(N)})dW_j, \quad \delta {x}_j(0) = 0. \\
 \end{aligned}
\end{equation}
{ The summation of \eqref{Eq12} yield that}
\begin{equation}\label{Eq12.5}
\resizebox{\textwidth}{!}{$
 \left\{
 \begin{aligned}
 & d \delta x^{( - i)} = \Big[A\delta x^{( - i)} + F\frac{N - 1}{N}(\delta x^{( - i)} + \delta x_i)\Big]dt + \sum_{j\neq i}\Big[C\delta x_{j} + \frac{\tilde{F}}{N}(\delta x^{( - i)} + \delta x_i)\Big]dW_j, \\
 & \delta x^{ (- i)}(0) = 0, \\
 \end{aligned}
 \right.$}
\end{equation}
where $\delta x^{( - i)}\triangleq\sum_{j\neq i}\delta{x}_j$.
Further, by some elementary calculations, one can  obtain the variation of $\mathcal{A}_i$ cost functional
\begin{equation}\label{Eq13}
 \begin{aligned}
 \delta \mathcal J_i = & \mathbb{E}\int_{0}^{T}\langle Q(\bar{x}_i - \Gamma \bar{x}^{(N)} - \eta), \delta x_i - \Gamma \delta x^{(N)}\rangle + \langle R\bar{u}_i, \delta u_i\rangle dt\\
 &+\mathbb{E}\langle G(\bar{x}_i(T) -\bar{\Gamma}\bar{x}^{(N)}(T) - \bar{\eta}),\delta x_i(T) - \bar{\Gamma} \delta x^{(N)}(T)\rangle,
 \end{aligned}
\end{equation}
and correspondingly, for $j\neq i$,
\begin{equation}\label{Eq14}
 \begin{aligned}
 \delta \mathcal J_j = &\mathbb{E}\int_{0}^{T}\langle Q(\bar{x}_j - \Gamma \bar{x}^{(N)} - \eta), \delta x_j - \Gamma \delta x^{(N)}\rangle dt\\
 &+\mathbb{E}\langle G(\bar{x}_i(T) -\bar{\Gamma}\bar{x}^{(N)}(T) - \bar{\eta}),\delta x_j(T) - \bar{\Gamma} \delta x^{(N)}(T)\rangle.
 \end{aligned}
\end{equation}
Hence, by combining \eqref{Eq13} and \eqref{Eq14}, the variation of the social cost functional satisfies
\begin{equation*}\resizebox{\textwidth}{!}{$
 \begin{aligned}
 &\delta \mathcal J_{soc}^{(N)} \triangleq \delta \mathcal J_i + \sum_{j\neq i}\delta \mathcal J_j \\
  = & \mathbb{E}\bigg\{\int_{0}^{T}\langle Q(\bar{x}_i - \Gamma \bar{x}^{(N)} - \eta), \delta x_i\rangle - \langle \Gamma^TQ(\bar{x}_i - \Gamma \bar{x}^{(N)} - \eta), \delta x^{(N)}\rangle + \sum_{j\neq i}\langle Q(\bar{x}_j- \Gamma \bar{x}^{(N)}- \eta),\\
 & \delta x_j\rangle - \sum_{j\neq i}\langle \Gamma^TQ(\bar{x}_j - \Gamma \bar{x}^{(N)} - \eta), \delta x^{(N)}\rangle+ \langle R\bar{u}_i, \delta u_i\rangle dt+ \langle G(\bar{x}_i(T)-\bar{\Gamma}\bar{x}^{(N)}(T) - \bar{\eta}),\\
 &\delta x_i(T)\rangle - \langle \bar{\Gamma}^TG(\bar{x}_i(T)-\bar{\Gamma}\bar{x}^{(N)}(T) - \bar{\eta}), \delta x^{(N)}(T)\rangle + \sum_{i\neq j}^{}\langle G(\bar{x}_j(T)-\bar{\Gamma}\bar{x}^{(N)}(T) - \bar{\eta}),\delta x_j(T)\rangle\\
 &- \sum_{i\neq j}^{} \langle \bar{\Gamma}^TG(\bar{x}_j(T)-\bar{\Gamma}\bar{x}^{(N)}(T) - \bar{\eta}), \delta x^{(N)}(T)\rangle \bigg\}.
 \end{aligned} $}
 \end{equation*}
We use mean-field term $\hat{x}$ to replace $\bar{x}^{(N)}$. Thus, we have
 \begin{equation*}\resizebox{\textwidth}{!}{$
 \begin{aligned}
&\delta \mathcal J_{soc}^{(N)}
 = \mathbb{E}\bigg\{\int_{0}^{T}\langle Q(\bar{x}_i - \Gamma \hat{x} - \eta), \delta x_i\rangle + \sum_{j\neq i}\frac{1}{N}\langle Q(\bar{x}_j - \Gamma \hat{x} - \eta), N\delta x_j\rangle - \langle \frac{1}{N} \sum_{j\neq i} \Gamma^TQ(\bar{x}_j\\
 & - \Gamma \hat{x} - \eta), \delta x_i + \delta x^{( - i)}\rangle + \langle R\bar{u}_i, \delta u_i\rangle dt
  +\langle G(\bar{x}_i(T)-\bar{\Gamma}\hat{x}(T) - \bar{\eta}),\delta x_i(T)\rangle +\sum_{i\neq j}^{}\frac{1}{N}\langle G(\bar{x}_j(T)\\
 &-\bar{\Gamma}\hat{x}(T) - \bar{\eta}),N\delta x_j(T)\rangle
  - \sum_{i\neq j}^{}\frac{1}{N} \langle \bar{\Gamma}^TG(\bar{x}_j(T)-\bar{\Gamma}\hat{x}(T) - \bar{\eta}), \delta x_i(T) + \delta x^{( - i)}(T)\rangle+ \varepsilon_1 + \varepsilon_2\bigg\}, \\
 \end{aligned} $}
\end{equation*}
where
\begin{equation}\label{Eq15}\resizebox{\textwidth}{!}{$
\left\{
 \begin{aligned}
 &\varepsilon_1 = \mathbb{E}\bigg\{\int_{0}^{T}\langle(\Gamma^TQ \Gamma -Q \Gamma) (\bar{x}^{(N)} - \hat{x}) , N\delta x^{(N)}\rangle dt
 +\langle(\bar{\Gamma}^TQ \bar{\Gamma} -Q \bar{\Gamma}) (\bar{x}^{(N)}(T) - \hat{x}(T)) , N\delta x^{(N)}(T)\rangle\bigg\},\\
 &\varepsilon_2 =\mathbb{E}\bigg\{\int_{0}^{T}-\langle \Gamma^TQ(\bar{x}_i - \Gamma\bar{x}^{(N)} - \eta), \delta x^{(N)}\rangle dt
 - \langle \bar{\Gamma}^TG(\bar{x}_i(T)-\bar{\Gamma}\hat{x}(T) - \bar{\eta}), \delta x^{(N)}(T)\rangle\bigg\}.\\
 \end{aligned}
 \right. $}
\end{equation}
For the next step, we introduce limiting processes $x^{**}$ and $x_j^{*}$ satisfying
\begin{equation}\label{Eq15.5}
 \left\{
 \begin{aligned}
 & dx^{**} = (Ax^{**} + F\delta x_i + Fx^{**})dt,\quad x^{**}(0) = 0, \\
 & dx_j^{*} = (Ax_j^{*} + F\delta x_i + Fx^{**})dt + (Cx_j^{*} + \tilde{F}\delta x_i + \tilde{F}x^{**})dW_j,\quad x_j^{*}(0) = 0,
 \end{aligned}
 \right.
\end{equation}
to substitute $\delta x^{( - i)}$ and $N\delta x_j$.
This implies
\begin{equation}\label{Eq16}
\resizebox{\textwidth}{!}{$
 \begin{aligned}
 & \delta \mathcal J_{soc}^{(N)}
 = \mathbb{E}\bigg\{\int_{0}^{T}\langle Q(\bar{x}_i - \Gamma \hat{x} - \eta), \delta x_i\rangle + \frac{1}{N}\sum_{j\neq i}\langle Q(\bar{x}_j - \Gamma \hat{x} - \eta),  x_j^{*}\rangle- \langle \Gamma^TQ((I - \Gamma) \hat{x} - \eta),\delta x_i\rangle \\
 & - \langle \Gamma^TQ(\hat{x} - \Gamma \hat{x} - \eta), x^{**} \rangle +\langle R\bar{u}_i, \delta u_i\rangle dt+\langle G(\bar{x}_i(T)-\bar{\Gamma}\hat{x}(T) - \bar{\eta}),\delta x_i(T)\rangle +\frac{1}{N}\sum_{j\neq i}\langle G(\bar{x}_j(T)\\
 &-\bar{\Gamma}\hat{x}(T) - \bar{\eta}), x_j^{*}(T)\rangle- \langle \bar{\Gamma}^TG((I-\bar{\Gamma})\hat{x}(T) - \bar{\eta}), \delta x_i(T)\rangle -\langle \bar{\Gamma}^TG((I-\bar{\Gamma})\hat{x}(T) - \bar{\eta}), x^{**}(T)\rangle  + \sum_{i=1}^{4}\varepsilon_i\bigg\},
 \end{aligned}$}
\end{equation}
where
\begin{equation}\label{Eq17}
\resizebox{\textwidth}{!}{$
\begin{aligned}
&\varepsilon_3 = \mathbb{E}\bigg\{\int_{0}^{T} \frac{1}{N}\sum_{j\neq i}\langle Q(\bar{x}_j - \Gamma \hat{x} - \eta), N\delta x_j - x_j^{*}\rangle - \frac{1}{N}\sum_{j\neq i}\langle \Gamma^TQ(\bar{x}_j - \Gamma \hat{x} - \eta),\delta x^{( - i)} - x^{**} \rangle dt\\
  & + \frac{1}{N}\sum_{j\neq i}[\langle G(\bar{x}_j(T) - \bar{\Gamma} \hat{x}(T) - \bar{\eta}), N\delta x_j(T) - x_j^{*}(T)\rangle
  - \langle \bar{\Gamma}^TG(\bar{x}_j(T) - \bar{\Gamma} \hat{x}(T) - \bar{\eta}),\delta x^{( - i)}(T) - x^{**}(T) \rangle]\bigg\},\\
&\varepsilon_4 = \mathbb{E}\bigg\{\int_{0}^{T} - \Big\langle\Gamma^TQ\Big(\frac{\sum_{j\neq i}\bar{x}_j}{N} - \hat{x}\Big) , \sum_{i=1}^{N}\delta x_i\Big\rangle dt  - \Big\langle\bar{\Gamma}^TG\Big(\frac{\sum_{j\neq i}\bar{x}_j(T)}{N} - \hat{x}(T)\Big) , \sum_{i=1}^{N}\delta x_i(T)\Big\rangle\bigg\}.
\end{aligned}$}
\end{equation}
$\varepsilon_1$-$\varepsilon_4$ are actually $o(1)$ order and the rigorous proof will show in Section 6.
\begin{remark}
{ Because of $F, \tilde{F}, C \neq 0$, the limiting processes $x^{**}$ and $x_j^{*}$ driven by \eqref{Eq15.5} are stochastic,  and such feature will bring many technical difficulties. Mainly, when using dual method in the next step to substitute $x^{**}$ and $x_j^{*}$, the processes average will enter the drift term of the dual process.  This makes the dual process complicated and also bring difficulties in error estimation.

  By contrast, in some previous social optima works,  the limiting processes $x^{**}$, $x_j^{*}$ are usually deterministic or even unnecessary.
  If $ \tilde{F} = C = 0$ (e.g., \cite{bjj2017}), then the dynamics of $\delta x^{( - i)}$ \eqref{Eq12.5} and $ \delta x_j$ \eqref{Eq12} becomes}
\begin{equation*}
 \left\{
 \begin{aligned}
& d\delta {x}_j = (A\delta {x}_j + \frac{F}{N}\delta x^{i} + \frac{F}{N}\delta x^{(-i)})dt, \quad \delta {x}_j(0) = 0, \\
 & d \delta x^{( - i)} = \Big[A\delta x^{( - i)} + F\frac{N - 1}{N}(\delta x^{( - i)} + \delta x_i)\Big]dt,\quad \delta x^{ (- i)}(0) = 0. \\
 \end{aligned}
 \right.
\end{equation*}
Clearly, $\delta {x}_j$ and $\delta x^{( - i)}$ are all deterministic and so are $x^{**}$, $x_j^{*}$. {Besides, if $F = \tilde{F} = 0$ }(e.g., \cite{hcm2012}), then $\delta x^{( - i)}$ and $ \delta x_j$ becomes
\begin{equation*}
 \left\{
 \begin{aligned}
 & d\delta {x}_j = A\delta {x}_jdt + C\delta {x}_jdW_j, \quad \delta {x}_j(0) = 0, \\
 & d \delta x^{( - i)} = A\delta x^{( - i)} dt + \sum_{j\neq i} C\delta x_{j} dW_j, \quad \delta x^{ (- i)}(0) = 0. \\
 \end{aligned}
 \right.
\end{equation*}
By the homogeneity, $\delta {x}_j = \delta x^{( - i)} = 0$ and no limiting approximation is needed.
\end{remark}
Next, we will substitute  $x^{**}$ and $x_j^{*}$ by dual method.  It is very important to construct an auxiliary control problem for investigating decentralized control in socially optimal problem. An auxiliary control problem is usually a standard LQ control problem (see \cite{hcm2012}, \cite{bjj2017}). Since \eqref{Eq16} contains $x^{**}$ and $x_j^{*}$, we have to use a duality procedure (see \cite[Chapter 3]{yz1999}) to break away $\delta \mathcal J_{soc}^{(N)}$ from the dependence on $x^{**}$ and $x_j^{*}$. To this end, we introduce the following two auxiliary equations
\begin{equation}\label{Eq18}
 \left\{
 \begin{aligned}
 & dy_1^j = - [ A^Ty_1^j + C^T\beta_1^j+Q(\bar{x}_j - \Gamma \hat{x} - \eta)]dt + \beta_1^jdW_j + \sum_{j^{'}\neq j}\beta_1^{j'}dW_{j'}, \\
 & dy_2 = - [(A + F)^Ty_2+ F^T\mathbb{E}y_1^j + \tilde{F}^T\mathbb{E}\beta_1^j -\Gamma^TQ(\hat{x} - \Gamma \hat{x} - \eta) ]dt,\\
 &y_1^j(T) = G(\bar{x}_j(T)-\bar{\Gamma}\hat{x}(T) - \bar{\eta}),\quad y_2(T) = -\bar{\Gamma}^TG((I- \bar{\Gamma}) \hat{x}(T) - \bar{\eta}).
 \end{aligned}
 \right.
\end{equation}
By applying It\^{o} formula, we have the following duality relations
\begin{equation}\label{Eq19}
 \begin{aligned}
  &\mathbb{E}\langle G(\bar{x}_j(T)-\bar{\Gamma}\hat{x}(T) - \bar{\eta}), x_j^{*}(T)\rangle
  = \mathbb{E}\langle y_1^j(T),x_j^{*}(T)\rangle - \mathbb{E}\langle y_1^j(0),x_j^{*}(0)\rangle \\
  =& \mathbb{E}\int_{0}^{T}\langle - Q(\bar{x}_j - \Gamma \hat{x} - \eta), x_j^{*}\rangle + \langle \tilde{F}^T\beta_1^j + F^Ty_1^j, x^{**}\rangle + \langle \tilde{F}^T\beta_1^j + F^Ty_1^j, \delta x_i \rangle dt,
 \end{aligned}
\end{equation}
and
\begin{equation}\label{Eq20}
 \begin{aligned}
 &\mathbb{E}\langle -\bar{\Gamma}^TG((I- \bar{\Gamma}) \hat{x}(T) - \bar{\eta}), x^{**}(T)\rangle
 = \mathbb{E}\langle y_2(T),x^{**}(T)\rangle - \mathbb{E}\langle y_2(0),x^{**}(0)\rangle\\
 =& \mathbb{E}\int_{0}^{T}\langle \Gamma^TQ(\hat{x} - \Gamma \hat{x} - \eta) - F^T\mathbb{E}y_1^j - \tilde{F}^T\mathbb{E}\beta_1^j, x^{**}\rangle + \langle F^Ty_2, \delta x_i\rangle dt. \\
 \end{aligned}
\end{equation}
Combining \eqref{Eq16}, \eqref{Eq19} and \eqref{Eq20},
\begin{equation}\label{Eq21}\resizebox{\textwidth}{!}{$
 \begin{aligned}
 & \delta \mathcal J_{soc}^{(N)} = \mathbb E\int_{0}^{T} \langle Q\bar{x}_i, \delta x_i\rangle + \langle R\bar{u}_i, \delta u_i\rangle - \langle Q(\Gamma \hat{x} + \eta) + \Gamma^TQ((I - \Gamma) \hat{x} - \eta)- F^Ty_2 - \tilde{F}^T\mathbb{E}\beta_1^j\\
 & - F^T\mathbb{E}y_1^j, \delta x_i\rangle dt+\langle G\bar{x}_i(T),\delta x_i(T)\rangle
 -\langle G(\bar{\Gamma}\hat{x}(T) + \bar{\eta})+ \bar{\Gamma}^TG((I-\bar{\Gamma})\hat{x}(T) - \bar{\eta}), \delta x_i(T)\rangle + \sum_{i=1}^{5}\varepsilon_i,
 \end{aligned} $}
\end{equation}
where
\begin{equation}\label{Eq22}
  \varepsilon_5 =  \mathbb E\int_{0}^{T} \bigg\langle \tilde{F}^T\bigg(\mathbb{E}\beta_1^j - \frac{\sum_{j\neq i}^{N}\beta_1^j}{N}\bigg) + F^T\bigg(\mathbb{E}y_1^j - \frac{\sum_{j\neq i}^{N}y_1^j}{N}\bigg), \delta x_i\bigg\rangle dt.
\end{equation}
\begin{remark}
{In the discussion above, we introduce $N+1$ adjoint processes to break away $\delta \mathcal J_{soc}^{(N)}$ from the dependence on $x^{**}$ and $x_j^{*}$. This difficulty is brought by $F, \tilde{F}\neq 0$. By contrast, if  $F = \tilde{F} = 0$ (e.g., \cite{hcm2012}), then \eqref{Eq21} becomes
  \begin{equation*}
  \resizebox{\textwidth}{!}{$
 \begin{aligned}
 & \delta \mathcal J_{soc}^{(N)} = \mathbb E\int_{0}^{T} \langle Q\bar{x}_i, \delta x_i\rangle + \langle R\bar{u}_i, \delta u_i\rangle - \langle Q(\Gamma \hat{x} + \eta) + \Gamma^TQ((I - \Gamma) \hat{x} - \eta), \delta x_i\rangle dt\\
 & \hspace{2.3cm}+\langle G\bar{x}_i(T),\delta x_i(T)\rangle-\langle G(\bar{\Gamma}\hat{x}(T) + \bar{\eta})+ \bar{\Gamma}^TG((I-\bar{\Gamma})\hat{x}(T) - \bar{\eta}), \delta x_i(T)\rangle +  \varepsilon.
 \end{aligned}$}
\end{equation*}
Clearly, no additional adjoint process is needed to derive the auxiliary problem.

Moreover, due to $F, \tilde{F} \neq 0$, $\varepsilon_5$ is the difference between the average {and  expectation of $y_1^j$, $\beta_1^j$}. Generally, the dynamics of adapted terms $\beta_1^j$ is inaccessible. To estimate it, some decoupling method is applied via two Lyapunov equations in Section 6.}
\end{remark}

Therefore, by using mean-field term $\hat{y}_2$, $\hat{y}_1$, $\hat{\beta}_1$ to replace $y_2$, $\mathbb{E}y_1^j$, $\mathbb{E}\beta_1^j$ respectively,  { the variation of decentralized auxiliary cost functional  $\delta J_i$ yield that}
\begin{equation}\label{Eq23}\resizebox{\textwidth}{!}{$
 \begin{aligned}
 & \delta J_i = \mathbb{E}\int_{0}^{T} \langle Q\bar{\alpha}_i, \delta \alpha_i\rangle + \langle R\bar{v}_i, \delta v_i\rangle- \langle Q(\Gamma \hat{x} + \eta) + \Gamma^TQ[(I - \Gamma) \hat{x} - \eta]- F^T\hat{y}_2 - \tilde{F}^T \hat{\beta}_1\\
 & - F^T\hat{y}_1, \delta \alpha_i\rangle dt
 + \langle G\bar{\alpha}_i(T), \delta \alpha_i(T)\rangle-\langle G(\bar{\Gamma} \hat{x}(T) + \bar{\eta}) + \bar{\Gamma}^TG[(I - \bar{\Gamma}) \hat{x}(T) - \bar{\eta}],\delta \alpha_i(T)\rangle,
 \end{aligned} $}
\end{equation}
and
\begin{equation*}
 \begin{aligned}
  &\delta \mathcal J_{soc}^{(N)}  =  \delta J_i + \sum_{i=1}^{6} \varepsilon_i,
  \end{aligned}
\end{equation*}
where
\begin{equation}\label{Eq24}\resizebox{\textwidth}{!}{$
 \begin{aligned}
  &\varepsilon_6 = \mathbb{E}\int_{0}^{T} \langle Q(\bar{x}_i - \bar{\alpha}_i), \delta x_i\rangle \!+\!  \langle Q \bar{\alpha}_i, \delta x_i - \delta \alpha_i\rangle \!+\! \langle R(\bar{u}_i - \bar{v}_i), \delta u_i\rangle \!+\! \langle R\bar{v}_i, \delta u_i-\delta v_i\rangle \\
  & + \langle  F^T(y_2-\hat{y}_2) + \tilde{F}^T(\mathbb{E}\beta_1^j-\hat{\beta}_1) + F^T(\mathbb{E}y_1^j-\hat{y}_i), \delta x_i\rangle+ \langle  F^T\hat{y}_2 + \tilde{F}^T \hat{\beta}_1 + F^T\hat{y}_1- Q\Gamma \hat{x}\\
  & - Q\eta - \Gamma^TQ[(I - \Gamma) \hat{x} - \eta],  \delta x_i - \delta \alpha_i\rangle dt+\langle G(\bar{x}_i(T) - \bar{\alpha}_i(T)), \delta x_i(T)\rangle+  \langle G \bar{\alpha}_i(T), \\
  & \delta x_i(T) - \delta \alpha_i(T)\rangle
  -\langle G(\bar{\Gamma} \hat{x}(T) + \bar{\eta}) + \bar{\Gamma}^TG[(I - \bar{\Gamma}) \hat{x}(T) - \bar{\eta}],\delta x_i(T) - \delta \alpha_i(T)\rangle.
  \end{aligned} $}
\end{equation}
Similarly, $\varepsilon_5$, $\varepsilon_6$ are also $o(1)$ order. Thus if $\delta J_i = 0$, we have $\delta \mathcal J_{soc}^{(N)}\longrightarrow0$.

\section{Decentralized control design}
\subsection{Auxiliary problem}
Motivated by equation \eqref{Eq23}, for
achieving $\delta J_i = 0$, we introduce the following auxiliary control problem:
\begin{problem}\label{P3}
Minimize $J_i(v_i)$ over $v_i\in\mathcal{U}_i$, where
\begin{equation*}
 \left\{
 \begin{aligned}
 & J_i\triangleq\frac{1}{2}\mathbb{E}\bigg\{\int_{0}^{T} \|\alpha_i\|_Q^2 + \|v_i\|^2_R + 2\langle q_1, \alpha_i\rangle dt+\|\alpha_i(T)\|_G^2+2\langle q_2, \alpha_i(T)\rangle\bigg\} , \\
 & d{\alpha}_i = (A{\alpha}_i + B{v}_i + F\hat{x})dt + (C{\alpha}_i + D{v}_i + \tilde{F}\hat{x})dW_i,\quad{\alpha}_i(0) = \xi_0, \\
 \end{aligned}
 \right.
\end{equation*}
and
\[q_1\triangleq - Q(\Gamma \hat{x} + \eta) - \Gamma^TQ[(I - \Gamma) \hat{x} - \eta] + F^T\hat y_2 + F^T\hat{y}_1 + \tilde{F}^T\hat\beta_1,\]
\[q_2\triangleq - G(\bar{\Gamma} \hat{x}(T) + \bar{\eta}) - \bar{\Gamma}^TG[(I - \bar{\Gamma}) \hat{x}(T) - \bar{\eta}].\]
\end{problem}
This is a standard LQ control problem, where the mean-field terms $\hat{x}$, $\hat y_2$, $\hat y_1$, $\hat\beta_1$ will be determined by the CC system in Section 5.2. {By} referring \cite{yz1999}, we have the following result.
\begin{lemma}\label{Lem5.1}
Under the {\rm\textbf{(A\ref{A1})}-\textbf{(A\ref{A3})}}, the following Riccati equation
 \begin{equation}\label{Eq25}
 \resizebox{\textwidth}{!}{$
 \left\{
 \begin{aligned}
 &\dot{P} + PA + A^TP + C^TPC + Q - (PB + C^TPD )(R + D^TPD)^{ - 1}(B^TP + D^TPC) = 0,\\
 &P(T) = G.
 \end{aligned}
 \right.$}
 \end{equation}
 admits a unique regular solution such that $(R + D^TPD)\gg 0$, and for any given $\hat{x}$, $\hat y_2$, $\hat y_1$, $\hat\beta_1 \in L^1(s,T;\mathbb R^{n})$, the auxiliary control problem ({Problem \ref{P3}}) admits a unique feedback form optimal control $\bar{v}_i = \Theta_1\bar{\alpha}_i + \Theta_2$, where
 \begin{equation}\label{Eq26}
 \resizebox{\textwidth}{!}{$
 \begin{aligned}
 & \Theta_1 \triangleq - (R + D^TPD)^{ - 1}( B^TP + D^TPC),\quad   \Theta_2 \triangleq - (R + D^TPD)^{ - 1}( B^T\varphi + D^TP\tilde{F}\hat{x}),
 \end{aligned}$}
 \end{equation}
 and $P$ satisfies \eqref{Eq25} while $\varphi$ is the unique solution of
 \begin{equation}\label{Eq27}
 \left\{
 \begin{aligned}
 & \dot{\varphi} + [A^T - (PB + C^TPD)(R + D^TPD)^{ - 1}B^T]\varphi \\
 & \hspace{3mm} - [(PB + C^TPD)(R + D^TPD)^{ - 1}D^T - C^T]P\tilde{F}\hat{x} + PF\hat{x} + q_1 = 0, \\
 & \varphi(T) = q_2.
 \end{aligned}
 \right.
 \end{equation}
 $\bar{\alpha}_i$ is the corresponding optimal auxiliary state satisfying
 \begin{equation}\label{Eq28}
 \left\{
 \begin{aligned}
 & d\bar{\alpha}_i = \big[(A + B\Theta_1) \bar{\alpha}_i + B\Theta_2 + F\hat{x}\big]dt + \big[(C + D\Theta_1) \bar{\alpha}_i + D\Theta_2 + \tilde{F}\hat{x}\big]dW_i,\\
 &\bar{\alpha}_i(0) = \xi_0. \\
 \end{aligned}
 \right.
\end{equation}
\end{lemma}
\subsection{CC system}
Apply the decentralized control law \eqref{Eq26} to \eqref{Eq1}, {and let} $\tilde{x}_i$ be the realized state. By the limiting approximation, the conditions $\mathbb{E}\tilde{x}_i = \hat{x}$, $\mathbb{E}y_1^j = \hat{y}_1$, $\mathbb{E}\beta_1^j = \hat{\beta}_1$ should holds. Then, $\tilde{x}_i$ satisfies
\begin{equation}\label{Eq29}
\left\{
 \begin{aligned}
 & d\tilde{x}_i = (A\tilde{x}_i + B(\Theta_1\tilde{x}_i + \Theta_2) + F\tilde{x}^{(N)})dt + (C\tilde{x}_i + D(\Theta_1\tilde{x}_i + \Theta_2) + \tilde{F}\tilde{x}^{(N)})dW_i,\\
 &\tilde{x}_i(0) = \xi_0, \\
 \end{aligned}
 \right.
\end{equation}
and
\begin{equation*}
\resizebox{\textwidth}{!}{$
 \begin{aligned}
 & \Theta_1 = - (R + D^TPD)^{ - 1}( B^TP + D^TPC),\quad   \Theta_2 = - (R + D^TPD)^{ - 1}( B^T\varphi + D^TP\tilde{F}\mathbb{E}\tilde{x}_i).
 \end{aligned}$}
 \end{equation*}
 {Here}
 \begin{equation*}
 \left\{
 \begin{aligned}
 & \dot{\varphi} + [A^T - (PB + C^TPD)(R + D^TPD)^{ - 1}B^T]\varphi \\
 & \hspace{3mm} - [(PB + C^TPD)(R + D^TPD)^{ - 1}D^T - C^T]P\tilde{F}\mathbb{E}\tilde{x}_i + PF\mathbb{E}\tilde{x}_i\\
 &\hspace{3mm} - Q(\Gamma \mathbb{E}\tilde{x}_i + \eta) - \Gamma^TQ[(I - \Gamma) \mathbb{E}\tilde{x}_i - \eta] + F^T\hat y_2 + F^T\hat{y}_1 + \tilde{F}^T\hat\beta_1 = 0, \\
 & \varphi(T) = - G(\bar{\Gamma} \mathbb{E}\tilde{x}_i(T) + \bar{\eta}) - \bar{\Gamma}^TG[(I - \bar{\Gamma}) \mathbb{E}\tilde{x}_i(T) - \bar{\eta}].
 \end{aligned}
 \right.
 \end{equation*}
 The auxiliary processes \eqref{Eq18} becomes
 \begin{equation*}
 \left\{
 \begin{aligned}
 & dy_1^j = - [ A^Ty_1^j + C^T\beta_1^j+Q(\tilde{x}_j - \Gamma \mathbb{E}\tilde{x}_i - \eta)]dt + \beta_1^jdW_j + \sum_{j^{'}\neq j}\beta_1^{j'}dW_{j'}, \\
 & dy_2 = - [(A + F)^Ty_2+ F^T\mathbb{E}y_1^j + \tilde{F}^T\mathbb{E}\beta_1^j -\Gamma^TQ(\mathbb{E}\tilde{x}_i - \Gamma \mathbb{E}\tilde{x}_i - \eta) ]dt,\\
 &y_1^j(T) = G(\tilde{x}_j(T)-\bar{\Gamma}\mathbb{E}\tilde{x}_i(T) - \bar{\eta}),\quad y_2(T) = -\bar{\Gamma}^TG((I- \bar{\Gamma}) \mathbb{E}\tilde{x}_i(T) - \bar{\eta}).
 \end{aligned}
 \right.
\end{equation*}
By taking summation and expectation of \eqref{Eq29}, we have
\begin{equation*}
\left\{
 \begin{aligned}
 & d\tilde{x}^{(N)} = (A\tilde{x}^{(N)}  +  B(\Theta_1\tilde{x}^{(N)}  +  \Theta_2)  +  F\tilde{x}^{(N)})dt\\
 &\hspace{1.5cm}+  \frac{1}{N}\sum_{i=1}^{N}(C\tilde{x}_i  +  D(\Theta_1\tilde{x}_i +  \Theta_2)  +  \tilde{F}\tilde{x}^{(N)})dW_i, \quad \tilde{x}^{(N)}(0) = \xi_0,\\
 & d\mathbb{E}\tilde{x}_i = (A\mathbb{E}\tilde{x}_i   +  B(\Theta_1\mathbb{E}\tilde{x}   +  \Theta_2)  +  F\mathbb{E}\tilde{x}_i)dt, \quad \mathbb{E}\tilde{x}_i(0)=\xi_0.\\
 \end{aligned}
 \right.
\end{equation*}
On the other hand, by  basic property of It\^{o} integral, for $i\neq j$, we have
\[\resizebox{\textwidth}{!}{$\left<\int_{0}^{t}(C\tilde{x}_i  +  D(\Theta_1\tilde{x}_i +  \Theta_2)  +  \tilde{F}\tilde{x}^{(N)})dW_i,\int_{0}^{t}(C\tilde{x}_j  +  D(\Theta_1\tilde{x}_j +  \Theta_2)  +  \tilde{F}\tilde{x}^{(N)})dW_j\right>=0.$}\]
Thus, it follows that
\[\lim_{N\rightarrow+\infty}\mathbb{E}\bigg\|\frac{1}{N}\sum_{i=1}^{N}\int_{0}^{t}(C\tilde{x}_i  +  D(\Theta_1\tilde{x}_i +  \Theta_2)  +  \tilde{F}\tilde{x}^{(N)})dW_i\bigg\|^2 = 0,\ \text{in }L^2(0,T;\mathbb{R}^n).\]
This implies that $\tilde{x}^{(N)}\rightarrow\mathbb{E}\tilde{x}_i$  when $N\rightarrow\infty$, and correspondingly we denotes $\tilde{x}_i\rightarrow\check{x}$, $y_1^j\rightarrow \check y_1$, $y_2\rightarrow \check y_2$. By letting
\begin{equation}\label{Eq30}
\resizebox{\textwidth}{!}{$
 \left\{
 \begin{aligned}
 & \Pi_1 = A - B(R + D^TPD)^{ - 1}( B^TP + D^TPC),\quad \Pi_2 = F - B(R + D^TPD)^{ - 1}D^TP\tilde{F},\\
 & \Pi_3 = - B(R + D^TPD)^{ - 1}B^T,\quad \Pi_1' = C - D(R + D^TPD)^{ - 1}( B^TP + D^TPC),\\
 & \Pi_4 = (PB + C^TPD)(R + D^TPD)^{ - 1}D^TP\tilde{F} - C^TP\tilde{F} + PF + Q\Gamma + \Gamma^TQ(I - \Gamma),\\
 & \Pi_2' = \tilde{F} - D(R + D^TPD)^{ - 1}D^TP\tilde{F},\quad \Pi_3' = - D(R + D^TPD)^{ - 1}B^T,\\
 \end{aligned}
 \right.$}
 \end{equation}
 we have the following CC system:
 \begin{equation}\label{Eq31}
 \left\{
 \begin{aligned}
 & d\check{x} = (\Pi_1\check{x} + \Pi_2\mathbb{E}\check{x} + \Pi_3\check\varphi)dt + (\Pi_1'\check{x} + \Pi_2'\mathbb{E}\check{x} + \Pi_3'\check\varphi)dB(t), \\
 & d\check{\varphi} = ( - \Pi_1^T\check\varphi + \Pi_4\mathbb{E}\check{x} - F^T\check{y}_2 - F^T\mathbb{E}\check{y}_1 - \tilde{F}^T\mathbb{E}\check\beta_1 + Q\eta - \Gamma^TQ\eta )dt, \\
 & d\check y_1 = ( - Q\check{x} + Q\Gamma \mathbb{E}\check{x} - A^T\check y_1 - C^T\check\beta_1 + Q\eta )dt + \check\beta_1dB(t), \\
 & d\check{y}_2 = [\Gamma^TQ(I - \Gamma) \mathbb{E}\check{x} - F^T\mathbb{E}\check y_1 - \tilde{F}^T\mathbb{E}\check\beta_1 - (A + F)^T\check{y}_2 -\Gamma^TQ\eta]dt,\\
 &\check{x}(0) = \xi_0,\quad \check\varphi(T) = - G(\bar{\Gamma} \mathbb{E}\check{x}(T) + \bar{\eta}) - \bar{\Gamma}^TG[(I - \bar{\Gamma}) \mathbb{E}\check{x}(T) - \bar{\eta}],\\
 &\check{y}_1(T) = G(\check{x}(T)-\bar{\Gamma}\mathbb{E}\check{x}(T) - \bar{\eta}),\quad \check{y}_2(T) = -\bar{\Gamma}^TG((I- \bar{\Gamma}) \mathbb{E}\check{x}(T) - \bar{\eta}),
 \end{aligned}
 \right.
 \end{equation}
 where $B(t)$ is a generic Brownian motion. Due to the symmetry and decentralization, only one generic Brownian motion $B(t)$ is needed to characterize the CC system here. Moreover, the mean-field terms $\hat{x}$, $\hat y_1$, $\hat\beta_1$ can be determined by $\hat{x} = \mathbb{E}\check{x}$, $\hat y_1 = \mathbb{E}\check y_1$, $\hat\beta_1 = \mathbb{E}\check\beta_1$, $\hat y_2 = \mathbb{E}\check y_2 = \check y_2$.
\begin{remark}
 {CC system \eqref{Eq31}} is a highly coupled MF-FBSDEs. It is different {from}  a general CC system (e.g., \cite{bensoussan2016linear}, \cite{jm2017}, \cite{h2010}), since the adapted terms (i.e., $C^T\check\beta_1$, $\tilde{F}^T\mathbb{E}\check\beta_1$) enter the drift term. {In general situation, by taking expectation on the realized state, an ODEs type CC system would be obtained. However, in our system here, $C^T\check\beta_1$, $\tilde{F}^T\mathbb{E}\check\beta_1$ can not be determined. Thus, the dynamics of $\hat{x}$, $\hat{y}_1$ and $\hat{\beta}_1$ cannot be obtained directly, and we can only represent them in an embedded way.}
\end{remark}
\subsection{Solvability of CC system}
 {In what follows, to simplify CC system \eqref{Eq31}, we apply some decentralizing method, which is firstly proposed by  Jiongmin Yong in \cite{y2013}.  This method would bring us a ``global solution'', while the traditional contraction mapping (see \cite{bensoussan2016linear}, \cite{nie}) can only lead to a ``local  solution''. Here ``global solution" means the solution is admitted in the whole time duration $[0,T]$, and ``local solution" means the solution could only be admitted  in a small time duration $[T_0,T]$.

  To apply the decentralizing method, for the first step, we rewrite \eqref{Eq31} as follows}
\begin{equation}\label{Eq32}
 \left\{
 \begin{aligned}
 & dX = \left(A_1X + \bar{A}_1\mathbb{E}X + B_1{Y}\right)dt + \left(A_1'X + \bar{A}_1'\mathbb{E}X + B_1'{Y}\right)dB(t), \\
 & d{Y} = \left(A_2X + \bar{A}_2\mathbb{E}X + B_2{Y} + \bar{B}_2\mathbb{E}Y + C_2{Z} + \bar{C}_2\mathbb{E}Z + f\right)dt +{Z}dB(t), \\
 & X(0) = \bar\xi_0,\quad {Y}(T) = \bar{G}X(T) + \bar{G}'\mathbb{E}X(T) +{g},
 \end{aligned}
 \right.
\end{equation}
where
\begin{equation*}
\resizebox{\textwidth}{!}{$
 \left\{
 \begin{aligned}
 & X \triangleq \left(\begin{smallmatrix}
 \check x \\
 0 \\
 0
 \end{smallmatrix}\right),
 Y \triangleq \left(\begin{smallmatrix}
 \check\varphi \\
 \check y_1 \\
 \check{y}_2
 \end{smallmatrix}\right), Z \triangleq \left(\begin{smallmatrix}
 0 \\
 \beta_1 \\
 0
 \end{smallmatrix}\right),\bar\xi_0\triangleq \left(\begin{smallmatrix}
 \xi_0 \\
 0 \\
 0\\
 \end{smallmatrix}\right),\\
 &A_1 \triangleq \left(\begin{smallmatrix}
 \Pi_1 & 0 & 0 \\
 0 & 0 & 0 \\
 0 & 0 & 0
 \end{smallmatrix}\right), \bar{A}_1\triangleq \left(\begin{smallmatrix}
 \Pi_2 & 0 & 0 \\
 0 & 0 & 0 \\
 0 & 0 & 0
 \end{smallmatrix}\right), B_1\triangleq\left(\begin{smallmatrix}
 \Pi_3 & 0 & 0 \\
 0 & 0 & 0 \\
 0 & 0 & 0
 \end{smallmatrix}\right), A_1' \triangleq \left(\begin{smallmatrix}
 \Pi_1' & 0 & 0 \\
 0 & 0 & 0 \\
 0 & 0 & 0
 \end{smallmatrix}\right), \bar{A}_1'\triangleq \left(\begin{smallmatrix}
 \Pi_2' & 0 & 0 \\
 0 & 0 & 0 \\
 0 & 0 & 0
 \end{smallmatrix}\right),\\
 &B_1'\triangleq\left(\begin{smallmatrix}
 \Pi_3' & 0 & 0 \\
 0 & 0 & 0 \\
 0 & 0 & 0
 \end{smallmatrix}\right),
 A_2 \triangleq \left(\begin{smallmatrix}
 0 & 0 & 0\\
 -Q & 0 & 0\\
 0 & 0 & 0\\
 \end{smallmatrix}\right) , \bar{A}_2 \triangleq \left(\begin{smallmatrix}
 \Pi_4 & 0 & 0\\
 Q\Gamma & 0 & 0\\
 - \Gamma^TQ(I - \Gamma)& 0 & 0
 \end{smallmatrix}\right),\\
 &B_2 \triangleq \left(\begin{smallmatrix}
 -\Pi_1^T & 0 & F^T \\
 0 & -A^T & 0 \\
 0 & 0 & -(A+F)^T
 \end{smallmatrix}\right),
 \bar{B}_2 \triangleq \left(\begin{smallmatrix}
 0 & -F^T & 0 \\
 0 & 0 & 0\\
 0 & -F^T & 0
 \end{smallmatrix}\right),
  C_2 \triangleq \left(\begin{smallmatrix}
 0 & 0 & 0 \\
 0 & -C^T & 0 \\
 0 & 0 & 0
 \end{smallmatrix}\right), \bar{C}_2 \triangleq \left(\begin{smallmatrix}
 0 & -\tilde{F}^T & 0 \\
 0 & 0 & 0 \\
 0 & -\tilde{F}^T & 0
 \end{smallmatrix}\right),\\
 & f \triangleq \left(\begin{smallmatrix}
 Q\eta - \Gamma^TQ\eta \\
 Q\eta \\
 \Gamma^TQ\eta\\
 \end{smallmatrix}\right),
 \bar{G} \triangleq \left(\begin{smallmatrix}
 0 & 0 & 0\\
 G & 0 & 0\\
 0 & 0 & 0\\
 \end{smallmatrix}\right) ,
 \bar{G}'\triangleq \left(\begin{smallmatrix}
 [- G\bar{\Gamma}   - \bar{\Gamma}^TG(I - \bar{\Gamma})]& 0& 0\\
 -G\bar{\Gamma}& 0& 0\\
 -\bar{\Gamma}^TG(I- \bar{\Gamma}) & 0& 0\\
 \end{smallmatrix}\right),
 g\triangleq \left(\begin{smallmatrix}
  \bar{\Gamma}^TG \bar{\eta} - G \bar{\eta}\\
  -G \bar{\eta} \\
  \bar{\Gamma}^TG \bar{\eta}\\
 \end{smallmatrix}\right).
 \end{aligned}
 \right.$}
\end{equation*}
 By taking expectation of \eqref{Eq32}, one can obtain the following decentralized functions
\begin{equation}\label{Eq33}
\resizebox{\textwidth}{!}{$
 \left\{
 \begin{aligned}
 & d\mathbb{E}X = \left[(A_1 + \bar{A}_1)\mathbb{E}X + B_1\mathbb{E}{Y}\right]dt,\ \ \mathbb{E}X(0) = \bar{\xi}_0, \\
 & d(X-\mathbb{E}X) = \left[A_1(X-\mathbb{E}X) + B_1(Y-\mathbb{E}{Y})\right]dt \\
 & \hspace{2.5cm}+\left[A_1'(X-\mathbb{E}X) + (A_1'+\bar{A}_1')\mathbb{E}X + B_1'(Y-\mathbb{E}Y) + B_1'\mathbb{E}Y \right]dB(t), \\
 & d\mathbb{E}{Y} = \left[(A_2 + \bar{A}_2)\mathbb{E}X + (B_2 + \bar{B}_2)\mathbb{E}Y + (C_2 + \bar{C}_2)\mathbb{E}Z + f\right]dt , \\
 & d(Y-\mathbb{E}Y) = \left[A_2(X-\mathbb{E}X) + B_2(Y-\mathbb{E}Y) + C_2(Z-\mathbb{E}Z) \right]dt + {Z}dB(t),\\
 &(X-\mathbb{E}X)(0) = 0,\ \ \mathbb{E}{Y}(T) = (\bar{G}+\bar{G}')\mathbb{E}X(T)+{g},\ \ Y(T)-\mathbb{E}Y(T) = \bar{G}(X-\mathbb{E}X)(T).
 \end{aligned}
 \right.$}
\end{equation}
Motivated by \eqref{Eq33}, we introduce the following FBSDEs
\begin{equation}\label{Eq34}
 \left\{
 \begin{aligned}
 & dX_1 = \left[(A_1 + \bar{A}_1)X_1 + B_1Y_1\right]dt, \\
 & dX_2 = \left[A_1X_2 + B_1Y_2\right]dt+\left[A_1'X_2 + (A_1'+\bar{A}_1')X_1 + B_1'Y_2 + B_1'Y_1 \right]dB(t),\\
 & dY_1 = \left[(A_2 + \bar{A}_2)X_1 + (B_2 + \bar{B}_2)Y_1 + (C_2 + \bar{C}_2)\mathbb{E}Z + f\right]dt ,  \\
 & dY_2 = \left[A_2X_2 + B_2Y_2 + C_2(Z-\mathbb{E}Z) \right]dt + {Z}dB(t),  \\
 &X_1(0) = \bar{\xi}_0,\ X_2(0) = 0,\ Y_1(T) = (\bar{G}+\bar{G}')X_1(T)+{g},\ Y_2(T) = \bar{G}X_2(T).
 \end{aligned}
 \right.
\end{equation}
By comparing  \eqref{Eq32} and \eqref{Eq34}, we have the following result.
\begin{proposition}
 If
 \begin{equation}\label{Eq37}
   \det\left((0,I) \Phi(T,0) \binom{0}{I}\right)\neq 0,
 \end{equation}
 where $\Phi$ is the transmission matrix w.r.t.  $\left(\begin{smallmatrix}
   A_1 & B_1 \\
   A_2-\bar{G}A_1+ (B_2 - \bar{G}B_1)\bar{G} & B_2 - \bar{G}B_1
 \end{smallmatrix}\right)$, then the MF-FBSDEs system \eqref{Eq32} is equivalent to the FBSDEs system \eqref{Eq34}.
\end{proposition}
\begin{proof}
 Suppose that $(X,Y,Z)$ is the adapted solution of system \eqref{Eq32}. Let $X_1 = \mathbb{E}X$, $Y_1 = \mathbb{E}Y$, ${X_2} = X - \mathbb{E}X$, ${Y_2} = Y - \mathbb{E}Y$, and it is easy to verify that $X_1$, $Y_1$, ${X_2}$, ${Y_2}$ satisfy the system \eqref{Eq34}. Conversely, if $(X_1, Y_1, X_2, Y_2)$ is the solutions of system \eqref{Eq34}, then let $X = X_1 + X_2$, $Y = Y_1 + Y_2$, and we have
 \begin{equation*}
 \left\{
 \begin{aligned}
 & dX = \left(A_1X + \bar{A}_1X_1 + B_1{Y}\right)dt + \left(A_1'X + \bar{A}_1'X_1 + B_1'{Y}\right)dB(t), \\
 & d{Y} = \left(A_2X + \bar{A}_2X_1 + B_2{Y} + \bar{B}_2Y_1 + C_2{Z} + \bar{C}_2\mathbb{E}Z + f\right)dt +{Z}dB(t), \\
 & X(0) = \bar{\xi}_0,\quad {Y}(T) = \bar{G}X(T) + \bar{G}'X_1(T) +{g}.
 \end{aligned}
 \right.
 \end{equation*}
 Thus, we just need to verify that $X_1 = \mathbb{E}(X_1 + X_2)$ and $Y_1 = \mathbb{E}(Y_1 + Y_2)$. Considering the expectation of $X_2$ and $Y_2$, it follows that
 \begin{equation*}
 \resizebox{\textwidth}{!}{$
 \left\{
 \begin{aligned}
 & d\mathbb{E}X_2 = \left[A_1\mathbb{E}X_2 + B_1\mathbb{E}Y_2\right]dt,\\
 & d(\mathbb{E}Y_2- \bar{G}\mathbb{E}X_2) = \left[(A_2-\bar{G}A_1+ (B_2 - \bar{G}B_1)\bar{G})\mathbb{E}X_2 + (B_2 - \bar{G}B_1)(\mathbb{E}Y_2- \bar{G}\mathbb{E}X_2) \right]dt,\\
 & \mathbb{E}X_2(0) = 0, \quad(\mathbb{E}Y_2 - \bar{G}\mathbb{E}X_2)(T) = 0. \\
 \end{aligned}
 \right.$}
 \end{equation*}
 Clearly,  if
 \[\det\left((0,I) \Phi(T,0) \binom{0}{I}\right)\neq 0,\]
   then $\mathbb{E}X_2 = \mathbb{E}{Y_2} \equiv 0$. Besides, $X_1$ and $Y_1$ are deterministic, which implies $X_1 = \mathbb{E}(X_1 + X_2)$ and $Y_1 = \mathbb{E}(Y_1 + Y_2)$.
\end{proof}

To study the solvability of \eqref{Eq34}, we firstly rewrite it as following form
\begin{equation}\label{cc3}
 \left\{
 \begin{aligned}
 & d \tilde{X} = (\tilde A_1 \tilde{X} + \tilde B_1 \tilde{Y})dt + ({\tilde A_1'} \tilde{X} + {\tilde B_1'} \tilde{Y})dB(t),\quad  \tilde{X}(0) = \tilde\xi_0,\\
 & d \tilde{Y} = ( \tilde A_2 \tilde{X} + \tilde B_2 \tilde{Y} + \tilde C_2 \tilde{Z} + \tilde{\bar{C}}_2\mathbb{E} \tilde{Z} + \tilde{f})dt + \tilde{Z} dB(t), \quad\tilde{Y}(T)=\tilde{G}\tilde{X}+\tilde{g},\\
 \end{aligned}
 \right.
\end{equation}
where $ \tilde{X} = \binom{X_1}{X_2}$, $ \tilde{Y} = \binom{Y_1}{Y_2}$, $ \tilde A_1 = \left(\begin{smallmatrix}
 A_1+\bar{A}_1 & 0 \\
 0 & A_1
 \end{smallmatrix}\right)$, $ \tilde B_1 = \left(\begin{smallmatrix}
 B_1 & 0 \\
 0 & B_1
 \end{smallmatrix}\right)$, $ {\tilde A_1'} = \left(\begin{smallmatrix}
 0 & 0 \\
 A_1' + \bar A_1' & A_1'
 \end{smallmatrix}\right)$, $ {\tilde B_1'} = \left(\begin{smallmatrix}
 0 & 0 \\
 B_1' & B_1'
 \end{smallmatrix}\right)$,
${\tilde A_2} = \left(\begin{smallmatrix}
 A_2+\bar{A}_2 & 0 \\
 0 & A_2
 \end{smallmatrix}\right)$, $ {\tilde B_2} = \left(\begin{smallmatrix}
 B_2+\bar{B}_2 & 0 \\
 0 & B_2
 \end{smallmatrix}\right)$, $ {\tilde C_2} = \left(\begin{smallmatrix}
 0 & 0 \\
 0 & C_2
 \end{smallmatrix}\right)$, $\tilde{\bar{C}}_2 = \left(\begin{smallmatrix}
 0 & C_2+\bar{C}_2 \\
 0 & -C_2
 \end{smallmatrix}\right)$, $\tilde{G} = \left(\begin{smallmatrix}
 \bar{G} + \bar{G}' & 0 \\
 0 & \bar{G}
 \end{smallmatrix}\right)$, $ \tilde{f} = \binom{f}{0}$, $\tilde\xi_0 = \binom{\bar{\xi}_0}{0}$, $\tilde{g} = \binom{g}{0}$.
{Since \eqref{cc3} is a FBSDE, we use Riccati equation method to decouple it, and its solvability yield that:}
\begin{lemma}\label{lemma 3}
 Under the {\rm\textbf{(A\ref{A1})}-\textbf{(A\ref{A2})}}, if the following Riccati equation
 \begin{equation}\label{RE2}
 \left\{
 \begin{aligned}
 & - \dot{K} + \tilde{B}_2{K} + ( \tilde{C}_2+\tilde{\bar{C}}_2){K}{\tilde A_1'} + ( \tilde C_2+\tilde{\bar{C}}_2){K}{\tilde B_1'}{K} - {K} \tilde A_1 - {K} \tilde B_1 + \tilde A_2 = 0, \\
 & {K}(T) = \tilde{G},
 \end{aligned}
 \right.
\end{equation}
 admits a unique solution and condition \eqref{Eq37} holds, then the system \eqref{cc3} admits a unique solution, and equivalently, the CC system \eqref{Eq31} admits a unique solution.
\end{lemma}

\begin{proof}
 Let $ \tilde{Y} = K \tilde{X} + \kappa $. By It\^{o} formula, we have
\begin{equation*}
\resizebox{\textwidth}{!}{$
 \begin{aligned}
 & ( \tilde{A}_2 \tilde{X} + \tilde{B}_2(K \tilde{X} + \kappa) + \tilde{C}_2 \tilde{Z} + {\tilde{\bar{C}}_2}\mathbb{E} \tilde{Z} + \tilde{f})dt + \tilde{Z} dB(t) = d \tilde{Y}\\
 = & dK \times \tilde{X} + K \times d \tilde{X} + d\kappa \\
 = & dK \times \tilde{X} + (K \tilde{A}_1 \tilde{X} + K \tilde{B}_1 K \tilde{X} + K \tilde{B}_1 \kappa)dt + (K\tilde{A}_1' \tilde{X} + K\tilde{B}_1'K \tilde{X} + K\tilde{B}_1'\kappa)dB(t) + d\kappa.
 \end{aligned}$}
\end{equation*}
Comparing the diffusion term, one can obtain
\[\tilde{Z} = K\tilde{A}_1' \tilde{X} + K\tilde{B}_1'K \tilde{X} + K\tilde{B}_1'\kappa,\]
which implies
\[\tilde{\bar{C}}_2\mathbb{E}\tilde{Z} = \tilde{\bar{C}}_2 K\tilde{A}_1'\mathbb{E} \tilde{X} + \tilde{\bar{C}}_2 K\tilde{B}_1'K \mathbb{E} \tilde{X} + \tilde{\bar{C}}_2 K\tilde{B}_1'\kappa.\]
Comparing the drift term and taking expectation, we obtain
\begin{equation}\label{cc drift coefficients}
\resizebox{\textwidth}{!}{$
 \begin{aligned}
 & \tilde{A}_2\mathbb{E} \tilde{X} + \tilde{B}_2 K \mathbb{E} \tilde{X} + \tilde{B}_2\kappa + (\tilde{C}_2 + \tilde{\bar{C}}_2)K\tilde{A}_1'\mathbb{E} \tilde{X} + (\tilde{C}_2 + \tilde{\bar{C}}_2)K\tilde{B}_1'K \mathbb{E} \tilde{X} + (\tilde{C}_2 + \tilde{\bar{C}}_2)K\tilde{B}_1'\kappa + \tilde{f} \\
 = & \dot K \mathbb{E} \tilde{X} + K \tilde{A}_1\mathbb{E} \tilde{X} + K \tilde{B}_1 K \mathbb{E} \tilde{X} + K \tilde{B}_1\kappa + \dot{{\kappa}}.
 \end{aligned}$}
\end{equation}
By comparing the coefficients of \eqref{cc drift coefficients}, $K$ should be the solution of \eqref{RE2}
and $\kappa$ satisfies
\begin{equation}\label{kappa}
 \begin{aligned}
 & - \dot{\kappa} + \big[\tilde{B}_2 + (\tilde{C}_2 + \tilde{\bar{C}}_2)K \tilde{B}_1' - K \tilde{B}_1\big]\kappa +\tilde{f} = 0,\quad\kappa(T) = \tilde{g}.
 \end{aligned}
\end{equation}
Under \textbf{(A\ref{A1})}-\textbf{(A\ref{A2})}, by \cite{sy2014}, equation \eqref{kappa} always admits a unique solution. Thus, if the Riccati equation \eqref{RE2} admits a unique solution, then the system \eqref{cc3} also admits a unique solution. {Correspondingly}, the CC system \eqref{Eq31}  admits a unique solution as well.
\end{proof}

The Riccati equation \eqref{RE2} can be rewritten as follows
\begin{equation}\label{RE3}
 \begin{aligned}
 \dot K = \tilde{A}_2 + \tilde{B}_2 K - K(\tilde{A}_1 + \tilde{B}_1 K) + (\tilde{C}_2 + \tilde{\bar{C}}_2)K(\tilde{A}_1' + \tilde{B}_1' K),\quad K(T) = \tilde{G},
 \end{aligned}
\end{equation}
which is not a general symmetric Riccati equation. Thus, it is not always solvable on $[0,T]$. {However, explicit solutions can still be obtained in some reduced but nontrivial cases. For example,} by \cite[Theorem 5.3]{y2006}, we have the following proposition.
\begin{proposition}\label{prop1}
 If we let $\tilde{C}_2 + \tilde{\bar{C}}_2 = 0$ (i.e., $C = \tilde{F} = 0$) and the following condition hold:
\begin{equation}
 \bigg[(0,I)\Psi(T,t)\binom{0}{I}\bigg]^{-1}\in L^1(0,T;\mathbb{R}^{n\times n}),
\end{equation}
then the Riccati equation \eqref{RE3} admit a unique solution $K$ which is given by the following:
\begin{equation*}
 K = -\bigg[(0,I)\Psi(T,t)\binom{0}{I}\bigg]^{-1}(0,I)\Psi(T,t)\binom{I}{0},\quad t\in[0,T],
\end{equation*}
where $\Psi(t,s)$ is the fundamental matrix w.r.t. $\left(\begin{smallmatrix}
\tilde{A}_1 & \tilde{B}_1 \\
\tilde{A}_2 & \tilde{B}_2
\end{smallmatrix}\right)$.

\end{proposition}
Through the discussion above, we have studied the solvability of the CC system. Besides the Riccati equation method, the solvability of a FBSDE can still be studied through many other techniques, like four-step scheme (see \cite{yz1999}) monotone condition (see \cite{peng1999fully}) and contraction mapping method (see \cite{pardoux1999forward}). { However, it is not the main topic in this paper, and we would not make more discussion in detail. In what follows, for the sake of discussion simplicity, we introduce the following assumption directly.}
\begin{assumption}\label{a5}{\rm(\textbf{A4})}
  The CC system \eqref{Eq31} admit a unique solution $(\check{x}, \check{y}_1, \check{y}_2, \check{\beta}_1)$.
  \end{assumption}
By {Lemma \ref{Lem5.1}}, we have the following theorem.
\begin{theorem}\label{them1}
 Under {\rm\textbf{(A\ref{A1})}-\textbf{(A\ref{a5})}}, {Problem \ref{P2}} admits a feedback form mean-field decentralized control $\tilde{u}_i = \Theta_1\tilde{x}_i + \Theta_2$, where
 \begin{equation*}
 \resizebox{\textwidth}{!}{$
 \begin{aligned}
 \Theta_1 \triangleq - (R + D^TPD)^{ - 1}( B^TP + D^TPC),\ \Theta_2 \triangleq - (R + D^TPD)^{ - 1}( B^T\varphi + D^TP\tilde{F}\hat{x}),
 \end{aligned}$}
 \end{equation*}
 and $P$, $\varphi$ are the solution of
 \begin{equation*}
 \resizebox{\textwidth}{!}{$
 \left\{
 \begin{aligned}
 & \dot{P} + PA + A^TP + C^TPC + Q - (PB + C^TPD )(R + D^TPD)^{ - 1}(B^TP + D^TPC) = 0, \\
 & \dot{\varphi} + [A^T - (PB + C^TPD)(R + D^TPD)^{ - 1}B^T]\varphi \\
 & \hspace{3mm} - [(PB + C^TPD)(R + D^TPD)^{ - 1}D^T - C^T]P\tilde{F}\hat{x} + PF\hat{x} \\
 & \hspace{3mm} - Q(\Gamma \hat{x} + \eta) - \Gamma^TQ[(I - \Gamma) \hat{x} - \eta] + (F^T\hat y_2 + F^T\hat{y}_1 + \tilde{F}^T\hat\beta_1) = 0, \\
 &P(T) = G,\quad \varphi(T) = - G(\bar{\Gamma} \hat{x}(T) + \bar{\eta}) - \bar{\Gamma}^TG[(I - \bar{\Gamma}) \hat{x}(T) - \bar{\eta}].
 \end{aligned}
 \right.$}
 \end{equation*}
 Mean-field terms $\hat{x}$, $\hat{y}_2$, $\hat{y}_1$, $\hat{\beta}_1$ are determined by $\hat{x} = \mathbb{E}\check{x}$, $\hat y_2 = \check y_2$, $\hat y_1 = \mathbb{E}\check y_1$, $\hat\beta_1 = \mathbb{E}\check\beta_1$ and $(\check{x}, \check y_2, \check{y}_1, \check{\beta}_1)$ is the solution of CC system \eqref{Eq31}. $\tilde{x}_i$ is realized state  satisfying
 \begin{equation}\label{realized state}
 \left\{
 \begin{aligned}
 & d\tilde{x}_i = (A\tilde{x}_i \!+\! B(\Theta_1\tilde{x}_i \!+\! \Theta_2) \!+\! F\tilde{x}^{(N)})dt \!+\! (C\tilde{x}_i \!+\! D(\Theta_1\tilde{x}_i \!+\! \Theta_2) \!+\! \tilde{F}\tilde{x}^{(N)})dW_i, \\
 & \tilde{x}_i(0) = \xi_0, \\
 \end{aligned}
 \right.
\end{equation}
where $\tilde{x}^{(N)} \triangleq \frac{1}{N}\sum_{i=1}^{N}\tilde{x}_i $.
\end{theorem}
Through the discussion above, the mean-field decentralized control has been characterized in Theorem \ref{them1}. In what follows, the performance of this mean-field decentralized control will be studied. Specifically, the asymptotic optimality will be proved.

\section{Asymptotic social optimality }
In this section, we will prove that the mean-field decentralized control given by {Theorem \ref{them1}} is asymptotically optimal. We introduce a new generic approach, which is different from traditional mean-field games scheme (e.g., \cite{bensoussan2016linear}, \cite{nie}, \cite{h2010}), where the authors usually use the auxiliary cost functional as a bridge to obtain the asymptotic optimality. Also, our method is different from some reduced social optima models (e.g., \cite{hcm2012}, \cite{bjj2017}, \cite{wang2019social}), where the {optimality} loss can be calculated directly with completing square method. Specifically, we estimate the social optimality loss by studying the Fr\'{e}chet differential of the social cost functional.

Firstly, we
distinguish the realized state and the optimal state of auxiliary problem. Let $\tilde{u}\triangleq(\tilde{u}_1,\cdots,\tilde{u}_N)$ be the mean-field decentralized control given by {Theorem \ref{them1}}, where $\tilde{u}_i = \Theta_1\tilde{x}_i + \Theta_2$. The realized decentralized state $\tilde{x}_i$ satisfies \eqref{realized state} which depends on the state-average $\tilde{x}^{(N)}$, while the optimal auxiliary state $\bar{\alpha}_i$ satisfies \eqref{Eq28} which depends on the mean-field term $\hat{x}$, and the optimal auxiliary control is $\bar{v}_i = \Theta_1\bar{\alpha}_i + \Theta_2$.
Correspondingly, the social cost functional $\mathcal{J}_i$ w.r.t the mean-field decentralized control $\tilde{u}$ is
\begin{equation*}
\resizebox{\textwidth}{!}{$
 \begin{aligned} \mathcal{J}_i(\tilde{u}_i,\tilde{u}_{-i}) = \frac{1}{2}\mathbb{E}\bigg\{\int_{0}^{T}\|\tilde{x}_i - \Gamma \tilde x^{(N)} - \eta\|_{Q}^2 + \|\tilde{u}_i\|^2_{R} dt+\|\tilde{x}_i(T) - \bar{\Gamma} \tilde{x}^{(N)}(T) - \bar{\eta}\|_{G}^2\bigg\}, \\
 \end{aligned}$}
\end{equation*}
and the auxiliary cost functional ${J}_i$ w.r.t the optimal auxiliary control $\bar{v}_i$ is
\begin{equation*}
 \begin{aligned} {J}_i(\bar{v}_i) = \frac{1}{2}\mathbb{E}\bigg\{\int_{0}^{T}\|\bar{\alpha}_i - \Gamma \hat{x} - \eta\|_{Q}^2 + \|\bar{v}_i\|^2_{R} dt+\|\bar{\alpha}_i(T) - \bar{\Gamma} \hat{x}(T) - \bar{\eta}\|_{G}^2\bigg\}. \\
 \end{aligned}
\end{equation*}
Next, we present the definition of asymptotic optimality.
\begin{definition}
A decentralized control $u^\varepsilon\triangleq(u^\varepsilon_1,\cdots,u^\varepsilon_N)\in \mathcal{U}_1\times\cdots\times\mathcal{U}_N$ has asymptotic social optimality if
\begin{equation*}
 \frac{1}{N}\Big|\inf_{u\in \mathcal{U}_c}\mathcal{J}_{soc}^{(N)}({u}) - \mathcal{J}_{soc}^{(N)}({u}^\varepsilon)\Big| = \varepsilon(N),\quad \varepsilon(N)\rightarrow0 \text{, when } N\rightarrow\infty,
\end{equation*}
where $\mathcal{U}_c$ is defined in Section 2 as a set of centralized information-based control.
\end{definition}
\subsection{Preliminary estimations}
To study the asymptotic optimality, we first provide several prior lemmas which will play a significant role in future analysis. 
In the discussion below, for sake of notation simplicity, we will use $L$ to denote a generic constant whose value may change from line to line and only depend on the coefficients (i.e., $A$, $B$, $C$, $D$, $F$, $\tilde{F}$, $\Gamma$, $\eta$, $Q$, $R$, $\xi_0$). 
\begin{lemma}\label{lemma 4}
Under {\rm\textbf{(A\ref{A1})}-\textbf{(A\ref{a5})}}, there exists some constant $L$ such that
\begin{equation*}
\begin{aligned}
\sup_{1\leq i\leq N}\mathbb{E}\sup_{0\leq t\leq T}\|\tilde{x}_i(t)\|^2 \leq L,\quad
\sup_{0\leq t\leq T}\mathbb{E}\|\tilde x^{(N)}(t)\|^2 \leq L,\quad \mathbb{E}\int_{0}^{T}\|\tilde{u}_i\|^2dt\leq L,
\end{aligned}
\end{equation*}
where $\tilde{u}_i = \Theta_1\tilde{x}_i + \Theta_2$.
\end{lemma}
\begin{proof}
By referring \cite[Section 5]{nie}, we have the first and the second inequality. Based on the boundness of $\tilde{x}_i$ and $\varphi$, the third inequality could be obtained easily. The detailed proof is omitted here.
\end{proof}
Next, the estimation of the difference between the mean-field term $\hat{x}$ and the realized state-average $\tilde{x}^{(N)}$ is as follows.
\begin{lemma} \label{lemma 6}
Under {\rm\textbf{(A\ref{A1})}-\textbf{(A\ref{a5})}}, It follows that
 \begin{equation}
 \mathbb{E}\sup_{0\leq t \leq T}\|\tilde{x}^{(N)} - \hat{x}\|^2 = O\bigg(\frac{1}{N}\bigg).
 \end{equation}
\end{lemma}
\begin{proof}
See Appendix A.
\end{proof}
Based on Lemma \ref{lemma 6} we have the following estimation.

\begin{lemma}\label{lemma 7}
Under {\rm\textbf{(A\ref{A1})}-\textbf{(A\ref{a5})}}, for some constant $L$ such that
\begin{equation*}
\begin{aligned}
\sup_{1\leq i\leq N}\mathbb{E}\sup_{0\leq t \leq T}\|\tilde{x}_i - \bar{\alpha}_i\|^2<L,\quad
\sup_{1\leq i\leq N}\mathbb{E}\sup_{0\leq t \leq T}\|\tilde{x}_i - \hat{x}\|^2<L.
 \end{aligned}
\end{equation*}
\end{lemma}
\begin{proof}
See Appendix A.
\end{proof}
Moreover, based on Lemma \ref{lemma 4}, \ref{lemma 6}, \ref{lemma 7}, we have the following estimation of the social cost functional.
\begin{lemma}\label{lemma 8}
 Under {\rm\textbf{(A\ref{A1})}-\textbf{(A\ref{a5})}}, for some constant $L$ such that
 \[\mathcal{J}^{(N)}_{soc}(\tilde{u}_1,\cdots,\tilde{u}_N) \leq NL.\]
\end{lemma}
\begin{proof}
See Appendix A.
\end{proof}
{Since we are studying the asymptotic optimality of $\tilde{u}$, it is sufficient only to  consider those admissible control  perform better then $\tilde{u}$. Specifically, we only consider those admissible controls $\acute{u}$ satisfying
\begin{equation}\label{a3}
 \mathcal{J}_{soc}^{(N)}(\acute{u}) \leq \mathcal{J}^{N}_{soc}(\tilde{u})\leq NL.
\end{equation}
For these admissible controls satisfying \eqref{a3}, we have the following estimation.}
\begin{proposition}\label{coro 1}
Under {\rm\textbf{(A\ref{A1})}-\textbf{(A\ref{a5})}}, for any admissible control $\acute{u}\in\mathcal{U}_c$ satisfying \eqref{a3}, there exists a constant $L$ such that
 \begin{equation*}
 \begin{aligned}
 \sum_{i=1}^{N}\mathbb{E}\int_{0}^{T}\|\acute{u}_i\|^2dt \leq NL. \\
 \end{aligned}
 \end{equation*}
\end{proposition}
\begin{proof}
By \eqref{a3}, this result can be obtained forthrightly.
\end{proof}
Next, we introduce a last lemma for single agent perturbation. Consider an admissible control $(\tilde{u}_1,\cdots,\tilde{u}_{i-1},\acute{u}_i,\tilde{u}_{i+1},\cdots, \tilde{u}_N)\in\mathcal{U}_c$. Note that here all the agents apply the mean-field decentralized control law given by Theorem \ref{them1} except $\mathcal{A}_i$. Correspondingly the agent state is denoted by $(\acute{x}_1, \cdots, \acute{x}_N)$. We denote $\delta u_i \triangleq \acute{u}_i - \tilde{u}_i$ and correspondingly $\delta x_j \triangleq \acute{x}_j - \tilde{x}_j$, $\acute{x}^{(N)} \triangleq \frac{1}{N}\sum_{j=1}^{N}\acute{x}_j$, $\delta x^{(N)} \triangleq \acute{x}^{(N)} - \tilde{x}^{(N)}$ for $j=1,\cdots,N$. Similarly, $\delta\alpha_i \triangleq \acute{\alpha}_i - \bar{\alpha}_i$, where $\bar{\alpha}_i$ is the optimal auxiliary state \eqref{Eq28} and $\acute{\alpha}_i $ is the $\delta u_i$-perturbed auxiliary state satisfying
\begin{equation*}
 \begin{aligned}
 d\acute{\alpha}_i = (A\acute{\alpha}_i + B\acute{v}_i + F\hat{x})dt + (C\acute{\alpha}_i + D\acute{v}_i + \tilde{F}\hat{x})dW_i,\quad\acute{\alpha}_i(0) = \xi_0. \\
 \end{aligned}
\end{equation*}
The $\delta u_i$-perturbed auxiliary control $\acute{v}_i = \Theta_1\bar{\alpha}_i + \Theta_2 + \delta u_i$.
\begin{lemma}\label{lemma 11}
Under {\rm\textbf{(A\ref{A1})}-\textbf{(A\ref{a5})}}, for some constant $L$ and any admissible control with form $(\tilde{u}_1$, $\cdots$, $\tilde{u}_{i-1}$, $\acute{u}_i$, $\tilde{u}_{i+1}$, $\cdots$, $\tilde{u}_N)\in\mathcal{U}_c$ satisfying $\mathbb{E}\int_{0}^{T}\|\acute{u}_i\|^2dt<L$, it follows that
\begin{equation*}
 \begin{aligned}
 \mathbb{E}\sup_{0\leq t\leq T}\|\delta x_i - \delta{a}_i\|^2 = O\Big(\frac{1}{N^2}\Big).
 \end{aligned}
 \end{equation*}
\end{lemma}
\begin{proof}
See Appendix A.
\end{proof}

\subsection{Asymptotic optimality}
Next, we begin to estimate the optimality loss. Recalling Section 3, the large-population system can be rewritten as follows
\begin{equation}\label{eq 49}
 \left\{
 \begin{aligned}
 &\text{minimize: }\mathcal{J}^{(N)}_{soc}(\mathbf{u}) = \frac{1}{2}\mathbb{E}\bigg\{\int_{0}^{T} \mathbf{x}^T\mathbf{Q}\mathbf{x} + 2\mathbf{S}_1^T\mathbf{x} + N\eta^TQ\eta + \mathbf{u}^T\mathbf{R}\mathbf{u} dt\\
 &\hspace{4.7cm}+\mathbf{x}(T)^T\mathbf{G}\mathbf{x}(T) + 2\mathbf{S}_2^T\mathbf{x}(T) + N\bar{\eta}^TG\bar{\eta}\bigg\}, \\
 &\text{subject to: }\begin{aligned}
   &d\mathbf{x} = (\mathbf{A}\mathbf{x} + \mathbf{B}\mathbf{u})dt + \sum_{i = 1}^{N}(\mathbf{C}_i\mathbf{x} + \mathbf{D}_i\mathbf{u})dW_i, \quad \mathbf{x}(0) = {\Xi},\\
 \end{aligned}
 \end{aligned}
 \right.
\end{equation}
where $\mathbf{x}$, $\mathbf{u}$, $\mathbf{Q}$, $\mathbf{S}_1$, $\mathbf{S}_2$, $\mathbf{R}$, $\mathbf{G}$, $\mathbf{A}$, $\mathbf{B}$, $\mathbf{C}$, $\mathbf{D}$, $\Xi$ follow  \eqref{Eq5} and \eqref{Eq7}. For any given admissible $\mathbf{u}$, the state $\mathbf{x}$ can be determined by
\begin{equation}
\resizebox{\textwidth}{!}{$
 \begin{aligned}
 \mathbf{x}(t) = & \Phi(t)\Xi + \Phi(t)\int_{0}^{t} \Phi(s)^{-1}\big[(\mathbf{B} - \sum_{i = 1}^{N}\mathbf{C}_i\mathbf{D}_i)\mathbf{u}(s)\big]ds + \sum_{i = 1}^{N}\Phi(t) \int_{0}^{t}\Phi(s)^{-1}\mathbf{D}_i\mathbf{u}dW_i(s),
 \end{aligned}$}
\end{equation}
where
\begin{equation*}
 \begin{aligned}
 & d\Phi(t) = \mathbf{A}\Phi(t)dt + \sum_{i = 1}^{m}\mathbf{C}_i\Phi(t)dW_i, \quad \Phi(0) = I.
 \end{aligned}
\end{equation*}
Define the following operators
\begin{equation*}
\resizebox{\textwidth}{!}{$
 \left\{
 \begin{aligned}
 & (L\mathbf{u}(\cdot))(\cdot) \triangleq \Phi(\cdot)\Bigg\{\int_{0}^{\cdot}\Phi(s)^{-1}\big[(\mathbf{B} - \sum_{i = 1}^{N}\mathbf{C}_i\mathbf{D}_i)\mathbf{u}(s)\big]ds+ \sum_{i = 1}^{N} \int_{0}^{\cdot}\Phi(s)^{-1}\mathbf{D}_i\mathbf{u}dW_i(s)\Bigg\}, \\
 & \tilde{L}\mathbf{u}(\cdot) = (L\mathbf{u}(\cdot))(T),\quad \Gamma \Xi(\cdot) = \Phi(\cdot)\Phi^{-1}(0)\Xi,\quad \tilde{\Gamma}\Xi = (\Gamma \Xi)(T).
 \end{aligned}
 \right.$}
\end{equation*}
Correspondingly, $L^*$ is defined as the adjoint operator of $L$
(see \cite{yz1999}).
Given any admissible $\mathbf{u}$,  $\mathbf{x}$  can be represented as follows
\begin{equation*}
 \begin{aligned}
 & \mathbf{x}(\cdot) = (L\mathbf{u}(\cdot))(\cdot) + \Gamma \Xi(\cdot),\quad \mathbf{x}(T) = \tilde{L}\mathbf{u}(\cdot) + \tilde{\Gamma}\Xi,
 \end{aligned}
\end{equation*}
and the cost functional can be rewritten as
\begin{equation}
\begin{aligned}
2\mathcal{J}^{(N)}_{soc}(\mathbf{u}) =& \mathbb{E}\bigg\{\int_{0}^{T} \mathbf{x}^T\mathbf{Q}\mathbf{x} + 2\mathbf{S}_1^T\mathbf{x} + N\eta^TQ\eta + \mathbf{u}^T\mathbf{R}\mathbf{u} dt+\mathbf{x}(T)^T\mathbf{G}\mathbf{x}(T)\\
& + 2\mathbf{S}_2^T\mathbf{x}(T) + N\bar{\eta}^TG\bar{\eta}\bigg\}
\triangleq \langle M_2\mathbf{u}(\cdot),\mathbf{u}(\cdot)\rangle + 2\langle M_1,\mathbf{u}(\cdot)\rangle + M_0. \\
 \end{aligned}
\end{equation}
where
\begin{equation*}
\left\{
 \begin{aligned}
& M_2\triangleq L^*\mathbf{Q}L + \tilde{L}^*\mathbf{G}\tilde{L} + \mathbf{R}, \quad M_1\triangleq L^*(\mathbf{Q}\Gamma \Xi(\cdot) + \mathbf{S}_1)+\tilde{L}^*(\mathbf{G}\tilde{\Gamma} \Xi(\cdot)+\mathbf{S}_2),\\
& M_0\triangleq \langle \mathbf{Q}\Gamma \Xi(\cdot),\Gamma \Xi(\cdot)\rangle+2\langle\mathbf{S}_1 , \Gamma \Xi(\cdot)\rangle+2\langle\mathbf{S}_2 , \tilde{\Gamma} \Xi(\cdot)\rangle + TN\eta^TQ\eta + N\bar{\eta}^TG\bar{\eta}.
  \end{aligned}
  \right.
\end{equation*}
Note that, $M_2$ is a bounded self-adjoint linear operator. Let $\tilde{\mathbf{u}} = [\tilde{{u}}_1^T, \cdots, \tilde{{u}}_n^T]^T$ be the decentralized control given by {Theorem \ref{them1}}. Consider a perturbation: ${\mathbf{u}} = \tilde{\mathbf{u}} + \delta\mathbf{u}$. Then
\begin{equation}\label{50}
 \begin{aligned}
2\mathcal{J}^{(N)}_{soc}(\tilde{\mathbf{u}} + \delta\mathbf{u}) &= \langle M_2(\tilde{\mathbf{u}} + \delta\mathbf{u}),\tilde{\mathbf{u}} + \delta\mathbf{u}\rangle + 2\langle M_1,\tilde{\mathbf{u}} + \delta\mathbf{u}\rangle + M_0\\
 & = 2\mathcal{J}^{(N)}_{soc}(\tilde{\mathbf{u}} ) + 2\langle M_2\tilde{\mathbf{u}} + M_1 , \delta\mathbf{u}\rangle + o(\delta \mathbf{u}).\\
 \end{aligned}
\end{equation}
Here, $\langle M_2\tilde{\mathbf{u}} + M_1 , \cdot\rangle$ is the Fr\'{e}chet differential of ${J}^{(N)}_{soc}$ on $\tilde{\mathbf{u}}$.
By the linearity, we also have
\begin{equation}
 \begin{aligned}
 \mathcal{J}^{(N)}_{soc}(\tilde{\mathbf{u}} + \delta\mathbf{u}) = \mathcal{J}^{(N)}_{soc}(\tilde{\mathbf{u}} ) + \sum_{i=1}^{N}\langle M_2\tilde{\mathbf{u}} + M_1 , \delta\mathbf{u}_i\rangle + o(\delta \mathbf{u}), \\
 \end{aligned}
\end{equation}
where $\delta \mathbf{u}_i \triangleq (0^T,\cdots,0^T,\delta u_i^T, 0^T,\cdots,0^T)^T$. Specifically,
\begin{equation}\label{eq 60}
 \begin{aligned}
 \mathcal{J}^{(N)}_{soc}(\tilde{\mathbf{u}} + \delta\mathbf{u}_i) = \mathcal{J}^{(N)}_{soc}(\tilde{\mathbf{u}} ) + \langle M_2\tilde{\mathbf{u}} + M_1 , \delta\mathbf{u}_i\rangle + o(\delta \mathbf{u}_i). \\
 \end{aligned}
\end{equation}
For the estimation of $M_2\tilde{\mathbf{u}} + M_1$, we have the following Lemma.
\begin{lemma}\label{lemma 8.9}
  Under {\rm\textbf{(A\ref{A1})}-\textbf{(A\ref{a5})}}, for the mean-field decentralized control $(\tilde{u}_1$, $\cdots$, $\tilde{u}_N)$ given by {Theorem \ref{them1}}, we have
\begin{equation*}
  \|M_2\tilde{\mathbf{u}} + M_1\| = O\Big(\frac{1}{\sqrt{N}}\Big).
\end{equation*}
\end{lemma}
\begin{proof}
See Appendix B.
\end{proof}

Based on the discussion above, we can introduce the result of the asymptotic optimality.

\begin{theorem}\label{them2}
 Under {\rm\textbf{(A\ref{A1})}-\textbf{(A\ref{a5})}}, the mean-field decentralized control $\tilde{u}$ given by {Theorem \ref{them1}} has asymptotic social optimality such that
 \begin{equation*}
 \frac{1}{N}\Big|\inf_{{u}\in\mathcal{U}_c}\mathcal{J}_{soc}^{(N)}({u}) - \mathcal{J}_{soc}^{(N)}(\tilde{u})\Big| = O\Big(\frac{1}{\sqrt{N}}\Big).
\end{equation*}
\end{theorem}
\begin{proof}
 By the representation of \eqref{eq 49}, $  \mathcal{J}^{(N)}_{soc}(\tilde{\mathbf{u}}) - \mathcal{J}^{(N)}_{soc}(\tilde{\mathbf{u}} + \delta\mathbf{u}) = o(N)$ is aimed to prove, for any admissible control $\tilde{\mathbf{u}} + \delta\mathbf{u}$ satisfying condition \eqref{a3}. By \eqref{50}, the following relation can be obtained
 \begin{equation}\label{eq 61}
 \begin{aligned}
  0\leq&\mathcal{J}^{(N)}_{soc}(\tilde{\mathbf{u}} ) - \mathcal{J}^{(N)}_{soc}(\tilde{\mathbf{u}} + \delta\mathbf{u}) = - \Bigg(\sum_{i=1}^{N}\langle M_2\tilde{\mathbf{u}} + M_1 , \delta\mathbf{u}_i\rangle\Bigg) - \frac{1}{2}\ \langle M_2 \delta\mathbf{u}, \delta\mathbf{u}\rangle \\
 \leq & \sqrt{\sum_{i=1}^{N}\| M_2\tilde{\mathbf{u}} + M_1\|^2 \sum_{i=1}^{N}\|\delta\mathbf{u}_i\|^2} - \frac{1}{2}\ \langle M_2 \delta\mathbf{u}, \delta\mathbf{u}\rangle \\
 = & \sqrt{N\| M_2\tilde{\mathbf{u}} + M_1\|^2} \times O(\sqrt{N}) - \frac{1}{2} \langle M_2 \delta\mathbf{u}, \delta\mathbf{u}\rangle.
 \end{aligned}
 \end{equation}
 Thus, by Lemma \ref{lemma 8.9} and noting the convexity (i.e., $\frac{1}{2} \langle M_2 \delta\mathbf{u}, \delta\mathbf{u}\rangle\geq 0$),  we have
 \begin{equation*}
 \begin{aligned}
  \mathcal{J}^{(N)}_{soc}(\tilde{\mathbf{u}} ) - \mathcal{J}^{(N)}_{soc}(\tilde{\mathbf{u}} + \delta\mathbf{u}) = &   O\Big({\sqrt{N}}\Big).
 \end{aligned}
 \end{equation*}
 Theorem \ref{them2} follows.
\end{proof}

\section{A numerical example}
Next, we will show an example by letting $n=m=2$,
$A=\left(\begin{smallmatrix}
	0.9723 & 0.9707 \\
	0.7409 & 0.0118
\end{smallmatrix}\right)$,
$B=\left(\begin{smallmatrix}
	0.7310 & 0.7980 \\
	0.2814 & 0.6108
\end{smallmatrix}\right)$,
$F=\left(\begin{smallmatrix}
	0.2077 & 0.4383 \\
	0.5265 & 0.2515
\end{smallmatrix}\right)$,
$C=\left(\begin{smallmatrix}
	0.5469 & 0.9669 \\
	0.3363 & 0.8207
\end{smallmatrix}\right)$,
$D=\left(\begin{smallmatrix}
	0.9051 & 0.8551 \\
	0.8856 & 0.4914
\end{smallmatrix}\right)$,
$\tilde{F}=\left(\begin{smallmatrix}
	0.4969 & 0.5103 \\
	0.4094 & 0.2017
\end{smallmatrix}\right)$,
$\xi_0=\left(\begin{smallmatrix}
	0.1627\\
	0.6570
\end{smallmatrix}\right)$,
$\Gamma=\left(\begin{smallmatrix}
	0.7420 & 0.9669 \\
	0.2016 & 0.1553
\end{smallmatrix}\right)$,
$\eta=\left(\begin{smallmatrix}
	0.8740\\
	0.7733
\end{smallmatrix}\right)$,
$Q=\left(\begin{smallmatrix}
	0.1845 & 0 \\
	0 & 0.1785
\end{smallmatrix}\right)$ and
$R=\left(\begin{smallmatrix}
	0.6587 & 0 \\
	0 & 0.8763
\end{smallmatrix}\right)$. Time interval $T = 1$ and the population $N=1000$.

 Firstly, by applying Runge-Kutta methods to \eqref{RE2} and \eqref{kappa}, we can obtain $K$ and $\kappa$ respectively. Then using the relation $ \tilde{Y} = K \tilde{X} + \kappa $ and $\tilde{Z} = K\tilde{A}_1' \tilde{X} + K\tilde{B}_1'K \tilde{X} + K\tilde{B}_1'\kappa$, \eqref{cc3} can be solved and we can determine $\tilde{X}$, $\tilde{Y}$, $\tilde{Z}$. Correspondingly, $\hat{x}$, $\hat{y_1}$, $\hat{\beta_1}$ can also be obtained. By calculating \eqref{Eq25}, \eqref{Eq26} and \eqref{Eq27}, we have $P$, $\varphi$, $\Theta_1$, $\Theta_2$, and then the realized state $\tilde{x}_i$ would follows by \eqref{realized state}.

 Through the calculation above, we can obtain the two coordinates of the trajectories of realized state-average $\tilde{x}^{(N)}$ and mean-field term $\hat{x}$ as follows.

\begin{figure}[htp]
\begin{center}
\includegraphics[height=9cm,width=13cm]{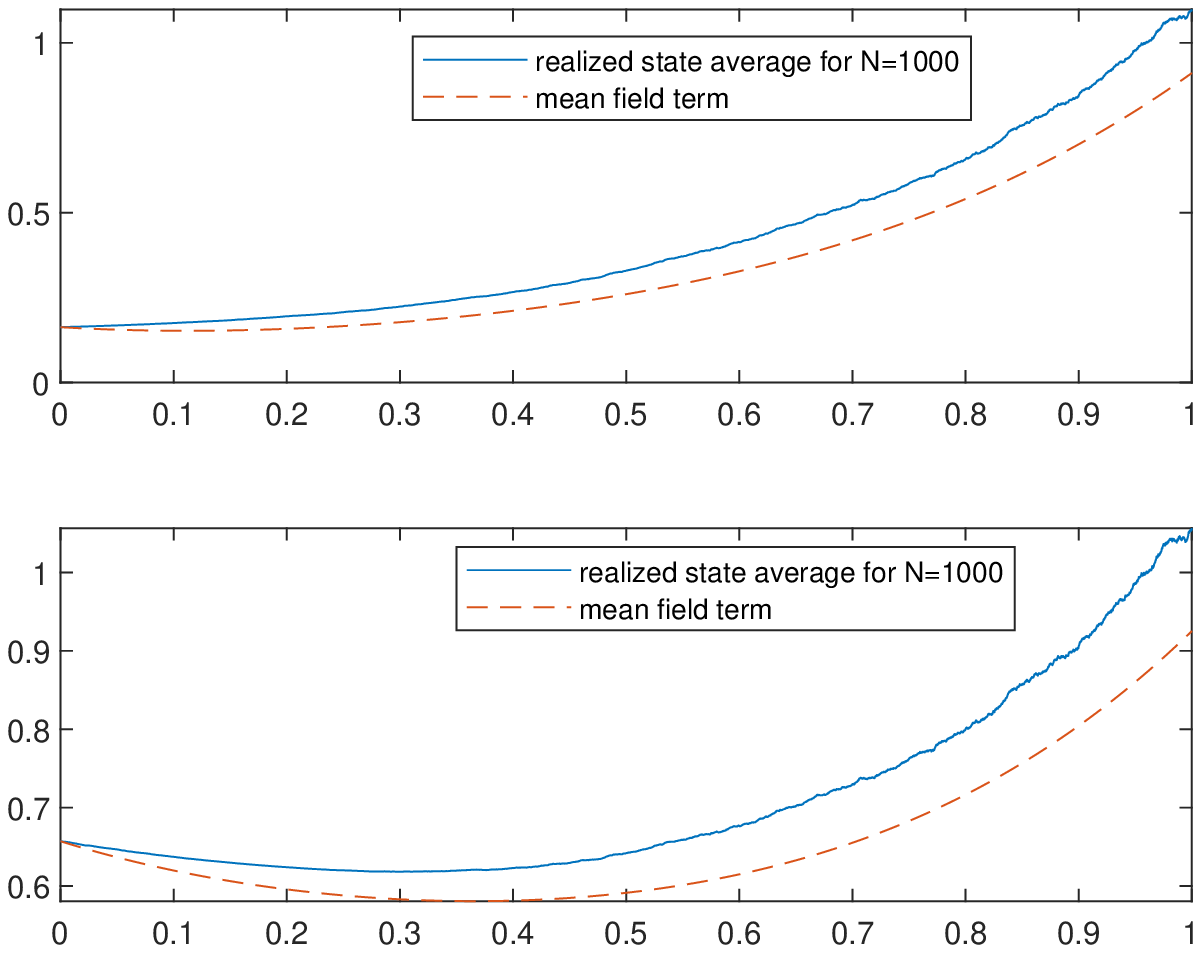}\\
  \end{center}
\end{figure}


\section{Conclusion}
In this paper, we investigate the social optimum of a generalized form of mean-field LQG control problem. First, we discuss the convexity of the social cost functional and summarize some conditions for some cases of indefinite weight coefficients. Then based on person by person optimality and duality procedures, a set of decentralized controls is designed by solving an auxiliary control problem which subject to a MF-FBSDEs system (CC system) in mean-field approximations. We study the well-posedness of the MF-FBSDEs system and obtain the conditions for its solvability. Finally, the corresponding decentralized controls is proved to has asymptotic social optimality.

\appendix
\section{Proofs of some preliminary estimations}
\begin{proof}[Proof of Lemma \ref{lemma 6}]
 By \eqref{Eq30}, \eqref{Eq31} and \eqref{realized state}, the dynamic of $\tilde{x}^{(N)}-\hat{x}$ follows
 \begin{equation*}
 \left\{
 \begin{aligned}
 & d(\tilde{x}^{(N)}-\hat{x}) = 
 (\Pi_1 + F)(\tilde{x}^{(N)}-\hat{x})dt + \frac{1}{N}\sum_{i=1}^{N}[(C + D\Theta_1 + \tilde{F})\tilde{x}_i + D\Theta_2]dW_i, \\
 & (\tilde{x}^{(N)}-\hat{x})(0) = 0.
 \end{aligned}
 \right.
 \end{equation*}
 By applying Cauchy-Schwartz inequality, Burkholder-Davis-Gundy's inequality and Lemma \ref{lemma 4}, there exist some constants $L$ such that
 \begin{equation*}
 \resizebox{\textwidth}{!}{$
 \begin{aligned}
 & \mathbb{E}\sup_{0\leq s \leq t}\|\tilde{x}^{(N)}(s) - \hat{x}(s)\|^2 \\
 \leq & 2\mathbb{E}\sup_{0\leq s\leq t}\left\|\int_{0}^{s}(\Pi_1 + F)(\tilde{x}^{(N)}-\hat{x})dr\right\|^2 + \frac{2}{N^2} \mathbb{E} \sup_{0\leq s\leq t} \left\|\sum_{i=1}^{N} \int_{0}^{s} [(C + D\Theta_1 + \tilde{F})\tilde{x}_i + D\Theta_2]dW_i\right\|^2 \\
 \leq & 2\mathbb{E}\int_{0}^{t}\left\|(\Pi_1 + F)(\tilde{x}^{(N)}-\hat{x})\right\|^2dr + \frac{L}{N^2}\mathbb{E}\sum_{i=1}^{N}\int_{0}^{t}\left\|(C + D\Theta_1 + \tilde{F})\tilde{x}_i + D\Theta_2\right\|^2dr \\
 = & L \mathbb{E}\int_{0}^{t}\left\|\tilde{x}^{(N)} - \hat{x}\right\|^2dr + O\Big(\frac{1}{N}\Big).
 \end{aligned}$}
 \end{equation*}
 Then, by Gronwall's inequality,
 \[\mathbb{E}\sup_{0\leq t \leq T}\|\tilde{x}^{(N)}(t) - \hat{x}(t)\|^2 = O\bigg(\frac{1}{N}\bigg).\]
\end{proof}

\begin{proof}[Proof of Lemma \ref{lemma 7}]
 By \eqref{realized state} and \eqref{Eq31}, we have
 \begin{equation*}\resizebox{\linewidth}{!}{$
 \left\{
 \begin{aligned}
 d(\tilde{x}_i - \hat{x}) 
 = & \Big\{\big[A - B(R + D^TPD)^{ - 1}( B^TP + D^TPC)\big](\tilde{x}_i-\hat{x}) + F(\tilde{x}^{(N)}-\hat{x}) \Big\}dt \\
 & + \big[(C + D\Theta_1)\tilde{x}_i + D\Theta_2 + \tilde{F}\tilde{x}^{(N)}\big]dW_i, \\
 \tilde{x}_i(0) - \hat{x}(0) & = 0.
 \end{aligned}
 \right.$}
 \end{equation*}
 By Cauchy-Schwarz inequality, Burkholder-Davis-Gundy's inequality and {Lemma \ref{lemma 6}}, there exist some constants $L$ such that
 \begin{equation*}\resizebox{\linewidth}{!}{$
 \begin{aligned}
 & \mathbb{E}\sup_{0\leq s\leq t}\|\tilde{x}_i(s) - \hat{x}(s)\|^2 \\
 \leq & 4\mathbb{E}\sup_{0\leq s\leq t}\int_{0}^{s}\Big\|\big[A - B(R + D^TPD)^{ - 1}( B^TP + D^TPC)\big](\tilde{x}_i - \hat{x})\Big\|^2 + \|F(\tilde{x}^{(N)}-\hat{x})\|^2dr \\
 & + 2\mathbb{E}\sup_{0\leq s\leq t}\Bigg\|\int_{0}^{s}\big[(C + D\Theta_1)\tilde{x}_i + D\Theta_2 + \tilde{F}\tilde{x}^{(N)}\big]dW_i\Bigg\|^2 \\
 \leq & L\Bigg\{\mathbb{E}\int_{0}^{t}\|\tilde{x}_i - \hat{x}\|^2dr + O\bigg(\frac{1}{N}\bigg) + \mathbb{E}\int_{0}^{t}\Big\|(C + D\Theta_1)\tilde{x}_i + D\Theta_2 + \tilde{F}\tilde{x}^{(N)}\Big\|^2dr\Bigg\} \\
 \leq & L\Bigg\{\mathbb{E}\int_{0}^{t}\|\tilde{x}_i - \hat{x}\|^2dr + O\bigg(\frac{1}{N}\bigg) + L\Bigg\}. \\
 \end{aligned}$}
 \end{equation*}
 Then, by Gronwall's inequality, one can obtain
 \[\mathbb{E}\sup_{0\leq t \leq T}\|\tilde{x}_i(t) - \hat{x}(t)\|^2\leq L,\]
 where $L$ is not related to $i$ and the lemma follows.
 Similarly, by \eqref{Eq28} and \eqref{realized state},  The dynamic of $\tilde{x}_i - \bar{\alpha}_i$ follows
  \begin{equation}
 \left\{
 \begin{aligned}
 & d(\tilde{x}_i - \bar{\alpha}_i) = [(A+ B\Theta_1)(\tilde{x}_i - \bar{\alpha}_i) + F(\tilde x^{(N)}-\hat{x})]dt\\
 &\hspace{2cm}+ [(C+ D\Theta_1)(\tilde{x}_i - \bar{\alpha}_i) + \tilde{F}(\tilde x^{(N)}-\hat{x})]dW_i, \\
 & (\tilde{x}_i - \bar{\alpha}_i)(0) = 0. \\
 \end{aligned}
 \right.
\end{equation}
Applying Burkholder-Davis-Gundy's inequality, Gronwall's inequality, and Lemma \ref{lemma 6}, the result can be obtained.
\end{proof}

\begin{proof}[Proof of Lemma \ref{lemma 8}]
 By \eqref{Eq2} and Lemma \ref{lemma 4}, \ref{lemma 6}, \ref{lemma 7}, for some constants $L$ we have
 \begin{equation*}
 \begin{aligned}
2\mathcal{J}^{(N)}_{soc}(\tilde{u}_1,\cdots,&\tilde{u}_N)
 =  \sum_{i = 1}^{N}\mathbb{E}\bigg\{\int_{0}^{T}\|\tilde{x}_i - \Gamma \hat{x} + \Gamma \hat{x} - \Gamma\tilde{x}^{(N)} -\eta\|_Q^2 + \|\tilde{u}_i\|^2_{R} \ dt \\
 &\hspace{1.5cm}+\|\tilde{x}_i(T)-\bar{\Gamma} \hat{x}(T)+\bar{\Gamma} \hat{x}(T) - \bar{\Gamma} \tilde{x}^{(N)}(T) - \bar{\eta}\|_{G}^2\bigg\}\\
 \leq& L\sum_{i = 1}^{N}\mathbb{E}\bigg\{\int_{0}^{T}\|\tilde{x}_i - \hat{x}\|^2 + \| \hat{x} - \tilde{x}^{(N)}\|^2 +\|\eta\|^2 + \|\tilde{u}_i\|^2 dt \\
 &\hspace{1.5cm}+\|\tilde{x}_i(T) - \hat{x}(T)\|^2 + \| \hat{x}(T) - \tilde{x}^{(N)}(T)\|^2 +\|\bar{\eta}\|^2\bigg\}\\
 \leq & L\sum_{i = 1}^{N}\mathbb{E}\int_{0}^{T}L + O\Big(\frac{1}{N}\Big) dt
 \leq  NL.
 \end{aligned}
 \end{equation*}
\end{proof}

To prove Lemma \ref{lemma 11}, we need the following lemmas.
\begin{lemma}\label{lemma 9}
Under {\rm\textbf{(A\ref{A1})}-\textbf{(A\ref{a5})}}, for some constant $L$ and any admissible control with form $(\tilde{u}_1$, $\cdots$, $\tilde{u}_{i-1}$, $\acute{u}_i$, $\tilde{u}_{i+1}$, $\cdots$, $\tilde{u}_N)\in\mathcal{U}_c$ satisfying $\mathbb{E}\int_{0}^{T}\|\acute{u}_i\|^2dt<L$, it follows that
\begin{equation*}
 \begin{aligned}
 \mathbb{E}\sup_{0\leq t\leq T}\|\delta x^{(N)}\|^2 = O\Big(\frac{1}{N^2}\Big).
 \end{aligned}
 \end{equation*}
\end{lemma}
\begin{proof}[Proof of Lemma \ref{lemma 9}]
 The dynamic of $\delta x^{(N)}$ follows
 \begin{equation*}\resizebox{\linewidth}{!}{$
 \left\{
 \begin{aligned}
 & d \delta x^{(N)} = \Big[(A + F)\delta x^{( N)} + \frac{B}{N}\delta u_i \Big]dt + \frac{1}{N}\sum_{j=1}^{N}(C\delta x_{j} + {\tilde{F}}\delta x^{( N)})dW_j + \frac{1}{N}D\delta {u}_idW_i, \\
 & \delta x^{ (N)}(0) = 0. \\
 \end{aligned}
 \right.$}
 \end{equation*}
 By {Lemma \ref{lemma 6}}, we have $\mathbb{E}\sup_{0\leq t \leq T}\|\tilde{x}^{(N)}(t) - \hat{x}(t)\|^2 = O\big(\frac{1}{N}\big)$. By Burkholder-Davis-Gundy's inequality and the boundness of $\acute{u}_i$, there exist some constants $L$ such that
 \begin{equation*}\resizebox{\linewidth}{!}{$
 \begin{aligned}
 \mathbb{E}\sup_{0\leq s\leq t}\|\delta x^{(N)}(s)\|^2 &= \mathbb{E}\sup_{0\leq s\leq t}\Bigg\|\int_{0}^{s}\Big[(A + F)\delta x^{( N)} + \frac{B}{N}\delta u_i \Big]dr +\frac{1}{N}D\int_{0}^{s}\delta {u}_idW_i\\
 &+ \frac{1}{N}\sum_{i=1}^{N}\int_{0}^{s}(C\delta x_{j} + {\tilde{F}}\delta x^{( N)})dW_j \Bigg\|^2 \\ \leq& L\mathbb{E}\int_{0}^{t}\|\delta x^{( N)}\|^2 dr+ \frac{L}{N^2}\sum_{i=1}^{N}\mathbb{E}\int_{0}^{t}\|\delta x_{i}\|^2 + \|\delta x^{( N)}\|^2dr + O\Big(\frac{1}{N^2}\Big).
 \end{aligned}$}
 \end{equation*}
 Applying Gronwall's inequality, for some constant $L$, one can obtain
 \begin{equation*}
 \begin{aligned}
 \mathbb{E}\sup_{0\leq s\leq t}\|\delta x_j(s)\|^2 \leq & L\mathbb{E}\int_{0}^{t}\|\delta x^{(N)}\|^2ds + \int_{0}^{t}L\mathbb{E}\int_{0}^{s}\|\delta x^{(N)}\|^2drds \\
 \leq & L\mathbb{E}\int_{0}^{t}\|\delta x^{(N)}\|^2ds, \quad j\neq i.\\
 \end{aligned}
 \end{equation*}
 Thus
 \begin{equation*}
 \begin{aligned}
 \mathbb{E}\sup_{0\leq s\leq t}\|\delta x^{(N)}(s)\|^2 \leq & L\mathbb{E}\int_{0}^{t}\|\delta x^{( N)}\|^2 ds + O\Big(\frac{1}{N^2}\Big).
 \end{aligned}
 \end{equation*}
 Again, by Gronwall's inequality
 \begin{equation*}
 \begin{aligned}
 \mathbb{E}\sup_{0\leq t\leq T}\|\delta x^{(N)}(t)\|^2 = O\Big(\frac{1}{N^2}\Big).
 \end{aligned}
 \end{equation*}
\end{proof}
\begin{lemma}\label{coro 2}
  Under {\rm\textbf{(A\ref{A1})}-\textbf{(A\ref{a5})}}, for some constant $L$ and any admissible control with form $(\tilde{u}_1$, $\cdots$, $\tilde{u}_{i-1}$, $\acute{u}_i$, $\tilde{u}_{i+1}$, $\cdots$, $\tilde{u}_N)\in\mathcal{U}_c$ satisfying $\mathbb{E}\int_{0}^{T}\|\acute{u}_i\|^2dt<L$, it follows that  \begin{equation}\label{sup delta xi}
 \sup_{1\leq j\leq N}\mathbb{E}\sup_{0\leq t\leq T}\|\delta x_j\|^2 = O\Big(\frac{1}{N^2}\Big).
\end{equation}
\end{lemma}
\begin{proof}
By the proof of Lemma \ref{lemma 9}, for some constant $L$, we have
  \begin{equation*}
 \mathbb{E}\sup_{0\leq s\leq t}\|\delta x_j\|^2 \leq L\mathbb{E}\int_{0}^{t}\|\delta x^{(N)}\|^2ds = O\Big(\frac{1}{N^2}\Big).
 \end{equation*}
 Note that $L$  is independent of $j$ and Proposition \ref{coro 2} holds.
\end{proof}
Based on Lemma \ref{lemma 9}, we can also obtain the following estimation of $\delta x_i - \delta{a}_i$.

\begin{proof}[Proof of Lemma \ref{lemma 11}]
The dynamics of $\delta x_i - \delta{a}_i$ follows
\begin{equation*}
\left\{
 \begin{aligned}
 & d(\delta x_i - \delta{a}_i) = [A(\delta x_i - \delta{a}_i) + F\delta x^{(N)} ]dt + [C(\delta x_i - \delta{a}_i) + \tilde{F}\delta x^{(N)}]dW_i,\\
 &(\delta x_i - \delta{a}_i)(0) = 0. \\
 \end{aligned}
 \right.
\end{equation*}
By Lemma \ref{lemma 9}, we have $\mathbb{E}\sup_{0\leq t\leq T}\|\delta x^{(N)}(t)\|^2 = O\big(\frac{1}{N^2}\big).$ Thus, applying  Burkholder-Davis-Gundy's inequality and Gronwall's inequality, the lemma follows.
\end{proof}

\section{Proof of Lemma \ref{lemma 8.9}}
\begin{proof}
Motivated by \eqref{eq 60},  we consider $\langle M_2\tilde{\mathbf{u}} + M_1, \delta\mathbf{u}_i\rangle$ for some {single-agent bounded perturbation} $\delta\mathbf{u}_i$ satisfying $\mathbb{E}\int_{0}^{T}\|\delta\mathbf{u}_i\|^2dt<L$ for some constant $L$. By the calculation in Section 3, when we only perturb $\mathcal{A}_i$, the variation of the cost functional is
 \begin{equation*}
 \resizebox{\textwidth}{!}{$
 \begin{aligned}
 &\mathcal{J}^{(N)}_{soc}(\tilde{\mathbf{u}} + \delta\mathbf{u}_i) = \mathcal{J}^{(N)}_{soc}(\tilde{\mathbf{u}} ) + \mathbb{E}\bigg\{\int_{0}^{T}\langle Q(\tilde{x}_i - \Gamma \tilde{x}^{(N)} - \eta), \delta x_i\rangle - \langle \Gamma^TQ(\tilde{x}_i - \Gamma \tilde{x}^{(N)} - \eta), \delta x^{(N)}\rangle \\
 & + \sum_{j\neq i}\langle Q(\tilde{x}_j - \Gamma \tilde{x}^{(N)} - \eta), \delta x_j\rangle - \sum_{j\neq i}\langle \Gamma^TQ(\tilde{x}_j - \Gamma \tilde{x}^{(N)} - \eta), \delta x^{(N)}\rangle+ \langle R\tilde{u}_i, \delta u_i\rangle dt + \langle G(\tilde{x}_i(T) \\
 &- \bar{\Gamma} \tilde{x}^{(N)}(T) - \bar{\eta}), \delta x_i(T)\rangle- \langle \bar{\Gamma}^TG(\tilde{x}_i(T) - \bar{\Gamma} \tilde{x}^{(N)}(T) - \bar{\eta}), \delta x^{(N)}(T)\rangle + \sum_{j\neq i}\langle G(\tilde{x}_j(T) \\
 &- \bar{\Gamma} \tilde{x}^{(N)}(T) - \bar{\eta}), \delta x_j(T)\rangle- \sum_{j\neq i}\langle \bar{\Gamma}^TG(\tilde{x}_j(T) - \bar{\Gamma} \tilde{x}^{(N)}(T) - \bar{\eta}), \delta x^{(N)}(T)\rangle\bigg\}+o(\delta \mathbf{u}_i).
 \end{aligned}$}
 \end{equation*}
 Thus, for the Fr\'{e}chet derivative of $\mathcal{A}_i$ on $\tilde{\mathbf{u}}$, we have
 \begin{equation}\label{eq 55}
 \resizebox{\textwidth}{!}{$
 \begin{aligned}
&\langle M_2\tilde{\mathbf{u}} + M_1 , \delta\mathbf{u}_i\rangle = \mathbb{E}\bigg\{\int_{0}^{T}\langle Q(\tilde{x}_i - \Gamma \tilde{x}^{(N)} - \eta), \delta x_i\rangle - \langle \Gamma^TQ(\tilde{x}_i - \Gamma \tilde{x}^{(N)} - \eta), \delta x^{(N)}\rangle\\
&+ \sum_{j\neq i}\langle Q(\tilde{x}_j - \Gamma \tilde{x}^{(N)} - \eta), \delta x_j\rangle - \sum_{j\neq i}\langle \Gamma^TQ(\tilde{x}_j - \Gamma \tilde{x}^{(N)} - \eta), \delta x^{(N)}\rangle + \langle R\tilde{u}_i, \delta u_i\rangle dt\\
& + \langle G(\tilde{x}_i(T)- \bar{\Gamma} \tilde{x}^{(N)}(T) - \bar{\eta}), \delta x_i(T)\rangle- \langle \bar{\Gamma}^TG(\tilde{x}_i(T) - \bar{\Gamma} \tilde{x}^{(N)}(T) - \bar{\eta}), \delta x^{(N)}(T)\rangle \\
 &+ \sum_{j\neq i}\langle G(\tilde{x}_j(T) - \bar{\Gamma} \tilde{x}^{(N)}(T) - \bar{\eta}), \delta x_j(T)\rangle- \sum_{j\neq i}\langle \bar{\Gamma}^TG(\tilde{x}_j(T) - \bar{\Gamma} \tilde{x}^{(N)}(T) - \bar{\eta}), \delta x^{(N)}(T)\rangle\bigg\}.
 \end{aligned}$}
 \end{equation}
 Next, we will verify $\varepsilon_1,\cdots,\varepsilon_6 = o(1)$, since
 \begin{equation}\label{eq 56}
   \langle M_2\tilde{\mathbf{u}} + M_1 , \delta\mathbf{u}_i\rangle - \sum_{i=1}^{6}\varepsilon_i = \delta J_i = 0.
 \end{equation}
  Equation \eqref{eq 56} follows by \eqref{Eq21}, \eqref{eq 55} and the optimality of the auxiliary cost functional. Firstly, we consider $\varepsilon_1$ which is given by \eqref{Eq15}.
By Lemma \ref{lemma 4}, Lemma \ref{lemma 6} and \ref{lemma 9}, for some constant $L$, we have
 \begin{equation}\resizebox{\textwidth}{!}{$
 \begin{aligned}
 \varepsilon_1 &= \mathbb{E}\bigg\{\int_{0}^{T}\langle(\Gamma^TQ \Gamma -Q \Gamma) (\tilde{x}^{(N)} - \hat{x}) , N\delta x^{(N)}\rangle dt \\
 &\hspace{1cm}+\langle(\bar{\Gamma}^TQ \bar{\Gamma} -Q \bar{\Gamma}) (\bar{x}^{(N)}(T) - \hat{x}(T)) , N\delta x^{(N)}(T)\rangle\bigg\}\\
 &\leq NL\sqrt{\mathbb{E}\int_{0}^{T}\|\tilde{x}^{(N)} - \hat{x}\|^2dt\mathbb{E}\int_{0}^{T} \| \delta x^{(N)}\|^2 dt}+O\Big(\frac{1}{\sqrt{N}}\Big)
 = O\Big(\frac{1}{\sqrt{N}}\Big),\\
\varepsilon_2 &=\mathbb{E}\bigg\{\int_{0}^{T}-\langle \Gamma^TQ(\bar{x}_i - \Gamma\bar{x}^{(N)} - \eta), \delta x^{(N)}\rangle dt - \langle \bar{\Gamma}^TG(\bar{x}_i(T)-\bar{\Gamma}\hat{x}(T) - \bar{\eta}), \delta x^{(N)}(T)\rangle\bigg\}\\
 &=2L\times O\Big(\frac{1}{N^2}\Big)= O\Big(\frac{1}{N^2}\Big).
 \end{aligned}$}
 \end{equation}


Next, we will estimate $\varepsilon_3$ which is given by \eqref{Eq17}. We consider $\delta x^{( - i)} - x^{**}$ first, where
 \begin{equation*}\resizebox{\textwidth}{!}{$
 \left\{
 \begin{aligned}
 & d(\delta x^{( - i)} - x^{**}) = \bigg[(A+F)(\delta x^{( - i)} - x^{**}) - \frac{F}{N}\delta x^{(N)} \bigg]dt
 + \sum_{j\neq i}\bigg[C\delta x_{j} + \frac{\tilde{F}}{N}(\delta x^{( - i)} + \delta x_i)\bigg]dW_j, \\
 & \delta x^{( - i)}(0) - x^{**}(0) = 0.
 \end{aligned}
 \right.$}
 \end{equation*}
 By  Lemma \ref{lemma 6} and Corollary \ref{coro 2}, for some constant $L$ such that
 \begin{equation*}\resizebox{\textwidth}{!}{$
 \begin{aligned}
 & \mathbb{E}\sup_{0\leq s\leq t}\|\delta x^{( - i)}(s) - x^{**}(s)\|^2\\
 \leq & L\mathbb{E}\int_{0}^{t}\bigg[\|\delta x^{( - i)} - x^{**}\|^2 + \frac{K}{N^2}\|\delta x^{(N)}\|^2 \bigg]ds
 + \mathbb{E}\sup_{0\leq s\leq t}\bigg\|\sum_{j\neq i}\int_{0}^{t}[C\delta x_{j} + \frac{\tilde{F}}{N}(\delta x^{( - i)} + \delta x_i)]dW_j\bigg\|^2 \\
 \leq & L\mathbb{E}\int_{0}^{t}\|\delta x^{( - i)} - x^{**}\|^2 ds + O\Big(\frac{1}{N^4}\Big) + L\mathbb{E}\sum_{j\neq i}\int_{0}^{t}\|\delta x_{j}\|^2 + \|\delta x^{( N)}\|^2ds \\
 = & L\mathbb{E}\int_{0}^{t}\|\delta x^{( - i)} - x^{**}\|^2 ds + O\Big(\frac{1}{N^4}\Big) + O\Big(\frac{1}{N}\Big) + O\Big(\frac{1}{N}\Big). \\
 \end{aligned}$}
 \end{equation*}
 Thus,
 \begin{equation*}
 \mathbb{E}\sup_{0\leq t\leq T}\|\delta x^{( - i)}(t) - x^{**}(t)\|^2 = O\Big(\frac{1}{N}\Big).
 \end{equation*}
 Similarly, the dynamics of $N\delta x_j - x^*_j$ is given by
 \begin{equation*}\resizebox{\textwidth}{!}{$
 \left\{
 \begin{aligned}
 d(N\delta x_j - x^*_j) = & \Big[A(N\delta x_j - x^*_j) + F(\delta x^{( - i)} - x^{**})\Big]dt+ \Big[C(N\delta x_j - x^*_j) + \tilde{F}(\delta x^{( - i)} - x^{**})\Big]dW_j, \\
 N\delta x_j(0) - x^*_j(0) & = 0.
 \end{aligned}
 \right.$}
 \end{equation*}
 For some constant $L$, one can obtain
 \begin{equation*}
 \begin{aligned}
 \mathbb{E}\sup_{0\leq s\leq t}\|N\delta x_j - x^*_j\|^2 = & L\mathbb{E}\int_{0}^{t} \|N\delta x_j - x^*_j\|^2 + \|\delta x^{( - i)} - x^{**}\|^2ds \\
 = & L\mathbb{E}\int_{0}^{t} \|N\delta x_j - x^*_j\|^2ds + O\Big(\frac{1}{N^2}\Big).
 \end{aligned}
 \end{equation*}
 Hence, by Gronwall's inequality,
 \begin{equation*}
 \mathbb{E}\sup_{0\leq t\leq T}\|N\delta x_j - x^*_j\|^2 = O\Big(\frac{1}{N^2}\Big).
 \end{equation*}
 For $\varepsilon_3$, there exist some constants $L$ such that
 \begin{equation*}\resizebox{\textwidth}{!}{$
 \begin{aligned}
 \varepsilon_3 = & \mathbb{E}\int_{0}^{T} \frac{1}{N}\sum_{j\neq i}\langle Q(\tilde{x}_j - \Gamma \hat{x} - \eta), N\delta x_j - x_j^{*}\rangle - \frac{1}{N}\sum_{j\neq i}\langle \Gamma^TQ(\tilde{x}_j - \Gamma \hat{x} - \eta),\delta x^{( - i)}\\
 &- x^{**} \rangle dt+ \frac{1}{N}\sum_{j\neq i}\langle G(\bar{x}_j(T) - \bar{\Gamma} \hat{x}(T) - \bar{\eta}), N\delta x_j(T) - x_j^{*}(T)\rangle -\frac{1}{N}\sum_{j\neq i}\langle \bar{\Gamma}^TG(\bar{x}_j(T)\\
 &- \bar{\Gamma} \hat{x}(T) - \bar{\eta}),\delta x^{( - i)}(T) - x^{**}(T) \rangle\\
 \leq & \frac{1}{N}\sum_{j\neq i}\sqrt{\mathbb{E}\int_{0}^{T}\| Q(\tilde{x}_j - \Gamma \hat{x} - \eta)\|^2dt \mathbb{E}\int_{0}^{T} \|N\delta x_j - x_j^{*}\|^2 dt} \\
 & + \frac{1}{N}\sum_{j\neq i}\sqrt{\mathbb{E}\int_{0}^{T}\| \Gamma^TQ(\tilde{x}_j - \Gamma \hat{x} - \eta)\|^2dt\mathbb{E}\int_{0}^{T}\|\delta x^{( - i)} - x^{**} \|^2 dt}+O\Big(\frac{1}{\sqrt{N}}\Big) \\
 = & \frac{1}{N}\sum_{j\neq i}\sqrt{L\times O\Big(\frac{1}{N^2}\Big)} + \sqrt{L\times O\Big(\frac{1}{N}\Big)} = O\Big(\frac{1}{\sqrt{N}}\Big).
 \end{aligned}$}
 \end{equation*}
In what follows, we aim prove that $ \varepsilon_4 = o(1)$. By Lemma \ref{lemma 6} and \ref{lemma 9} we have
 \begin{equation*}\resizebox{\textwidth}{!}{$
 \begin{aligned}
 	\varepsilon_4 = & \mathbb{E}\int_{0}^{T} - \Big\langle\Gamma^TQ\Big(\frac{\sum_{j\neq i}\tilde{x}_j}{N} - \hat{x}\Big) , \sum_{i=1}^{N}\delta x_i\Big\rangle dt  - \Big\langle\bar{\Gamma}^TG\Big(\frac{\sum_{j\neq i}\bar{x}_j(T)}{N} - \hat{x}(T)\Big) , \sum_{i=1}^{N}\delta x_i(T)\Big\rangle\\
 = & 2\times O\Big(\frac{1}{\sqrt{N}}\Big)\times O\Big(\frac{1}{{N}}\Big).
 \end{aligned}$}
 \end{equation*}
 For $\varepsilon_5 $, which is given by \eqref{Eq22}, we need to estimate $\mathbb{E}\beta_1^j - \frac{\sum_{j\neq i}^{N}\beta_1^j}{N}$ and $\mathbb{E}y_1^j - \frac{\sum_{j\neq i}^{N}y_1^j}{N}$. Recall that in Section 3, \eqref{Eq18} can be rewritten as follows
 \begin{equation*}
 \left\{
 \begin{aligned}
 & d\mathbf{y_1} = [ -\mathbf{\bar Q}{\tilde{\mathbf x}} + \mathbf{q} - \bar{\mathbf{A}}^T\mathbf{y_1} + \sum_{j = 1}^{N}\mathbf{\bar{C}_j}^T \mathbf{z_1^j}]dt + \sum_{j = 1}^{N}\mathbf{z_1^j}dW_j,\quad
 \mathbf{y_1}(T) = \mathbf{G}\tilde{\mathbf x}+\mathbf{g},\\
 & d{\tilde{\mathbf x}} = (\mathbf{A}{\tilde{\mathbf x}} + \mathbf{b})dt + \sum_{i = 1}^{N}(\mathbf{C}_i{\tilde{\mathbf x}} + \mathbf{h_i})dW_i, \quad \tilde{\mathbf{x}}(0) = {\Xi},\\
 \end{aligned}
 \right.
\end{equation*}
where
%
%
%
\begin{equation*}\resizebox{\linewidth}{!}{$
\bar{\mathbf{ A}} = \left(\begin{array}{cccc}
 A & 0 & \cdots & 0 \\
 0 & A & \cdots & 0 \\
 \vdots & \vdots & \ddots &\vdots\\
 0 & 0 & \cdots & A \\
 \end{array}\right),\
\mathbf{\bar{C}_j}=\begin{array}{c}
 1 \\
 \vdots \\
 j^{\text{th}} \\
 \vdots \\
 N
 \end{array}\left(\begin{array}{ccccc}
 0 & \cdots &  0  & \cdots & 0\\
 \vdots & \ddots & \vdots & \vdots & \vdots \\
 0 & \cdots &  C  & \cdots & 0\\
 \vdots & \vdots & \vdots & \ddots & \vdots\\
 0 & \cdots  & 0  & \cdots & 0\\
 \end{array}\right),\
\mathbf{\bar G} = \left(\begin{array}{cccc}
 G & 0 & \cdots & 0 \\
 0 & G & \cdots & 0 \\
 \vdots & \vdots & \ddots &\vdots\\
 0 & 0 & \cdots & G \\
 \end{array}\right),\
 \mathbf{\bar Q} = \left(\begin{array}{cccc}
 Q & 0 & \cdots & 0 \\
 0 & Q & \cdots & 0 \\
 \vdots & \vdots & \ddots &\vdots\\
 0 & 0 & \cdots & Q \\
 \end{array}\right),\ $}
\end{equation*}
\begin{equation*}\resizebox{\linewidth}{!}{$
\mathbf{A}=\left(\begin{array}{cccc}
 A + \frac{F}{N} + B\Theta_1 & \frac{F}{N} & \cdots & \frac{F}{N} \\
 \frac{F}{N} & A + \frac{F}{N} + B\Theta_1 & \cdots & \frac{F}{N} \\
 \vdots & \vdots & \ddots &\vdots\\
 \frac{F}{N} & \frac{F}{N} & \cdots & A + \frac{F}{N} + B\Theta_1\\
 \end{array}\right),\
\mathbf{b}=\left(\begin{array}{c}
 B\Theta_2 \\
 \vdots\\
 B\Theta_2
 \end{array}\right),\
 \Xi = \left(\begin{array}{c}
 \xi_0 \\
 \vdots\\
 \xi_0
 \end{array}\right),\
 \mathbf{y_1} = \left(\begin{array}{c}
 y_1^1 \\
 \vdots \\
 y_1^N
 \end{array}\right),\
 \tilde{\mathbf{x}}=\left(\begin{array}{c}
 \tilde{x}_1 \\
 \vdots\\
 \tilde{x}_N
 \end{array}\right),\ $}
 \end{equation*}
\begin{equation*}\resizebox{\linewidth}{!}{$
\mathbf{C}_i=\begin{array}{c}
 1 \\
 \vdots \\
 i^{\text{th}} \\
 \vdots \\
 N
 \end{array}\left(\begin{array}{ccccccc}
 0 & \cdots & 0 & 0 & 0 & \cdots & 0\\
 \vdots & & \ddots & \vdots & \vdots & \vdots & \vdots\\
 \frac{\tilde{F}}{N} & \cdots &\frac{\tilde{F}}{N} & \frac{\tilde{F}}{N} + C + D\Theta_1 &\frac{\tilde{F}}{N} &\cdots & \frac{\tilde{F}}{N}\\
 \vdots & \vdots & \vdots & \vdots & & \ddots & \vdots\\
 0 & \cdots & 0 & 0 & 0 & \cdots & 0\\
 \end{array}\right),\
\mathbf{h_i}=\begin{array}{c}
 1 \\
 \vdots \\
 i^{\text{th}} \\
 \vdots \\
 N
 \end{array}\left(\begin{array}{c}
 0\\
 \vdots\\
 D\Theta_2\\
 \vdots\\
 0 \\
 \end{array}\right),\
 \mathbf{g} = \left(\begin{array}{c}
 G\bar{\Gamma} \hat{x} + G\bar{\eta} \\
 \vdots \\
 G\bar{\Gamma} \hat{x} + G\bar{\eta}
 \end{array}\right),\
 \mathbf{q} = \left(\begin{array}{c}
 Q\Gamma \hat{x} + Q\eta \\
 \vdots \\
 Q\Gamma \hat{x} + Q\eta
 \end{array}\right),\
 \mathbf{z_1^j} = \left(\begin{array}{c}
 \beta_1^{1j} \\
 \vdots \\
 \beta_1^{Nj}
 \end{array}\right). $}
 \end{equation*}
Let $\mathbf{y_1} = \mathbf\Lambda \tilde{\mathbf{x}} + \mathbf{\lambda}$ and the following equations can be derived
\begin{equation}\label{Lambda}
\left\{
 \begin{aligned}
 &\mathbf{z_1^j} = \mathbf\Lambda\mathbf{C}_j\tilde{\mathbf{x}} + \mathbf\Lambda\mathbf{h_j}, \\
 & \dot{\mathbf\Lambda} + \mathbf\Lambda\mathbf{A} + \mathbf{\bar Q} + \bar{\mathbf{A}}^T\mathbf\Lambda - \sum_{j = 1}^{N}\mathbf{\bar{C}_j}^T \mathbf\Lambda\mathbf{C_j} = 0,\ \mathbf\Lambda(T)=\mathbf{G}, \\
 & \dot\lambda + \bar{\mathbf{A}}^T\lambda - \mathbf{q} - \sum_{j = 1}^{N}\mathbf{\bar{C}_j}^T \mathbf\Lambda\mathbf{h_j} + \mathbf\Lambda\mathbf{b} = 0,\ \lambda(T)=\mathbf{g},
 \end{aligned}
 \right.
\end{equation}
where
$\mathbf{\Lambda}=\left(\begin{smallmatrix}
  \Lambda_{11} & \cdots & \Lambda_{1N} \\
  \vdots & \ddots & \vdots \\
  \Lambda_{N1} & \cdots & \Lambda_{NN}
\end{smallmatrix}\right)$.
Thus, $\mathbb{E}y_1^j - \frac{\sum_{j=1}^{N}y_1^j}{N} = \frac{1}{N}(I,\cdots,I)\mathbf\Lambda \left(\begin{smallmatrix}
 \tilde{x}_1 - \mathbb{E}\tilde{x}_1\\
 \vdots \\
 \tilde{x}_N - \mathbb{E}\tilde{x}_N
 \end{smallmatrix}\right)$. Clearly, if there exist some constants $L$ such that $\displaystyle\sup_{1\leq i\leq N}\sup_{0\leq t\leq T}\|\sum_{k=1}^{N}\Lambda_{ki}(t)\|_{\max} < L$ holds, then by letting $E = \left(\begin{smallmatrix}
                               1 & \cdots & 1 \\
                               \vdots & \ddots & \vdots \\
                               1 & \cdots & 1
                             \end{smallmatrix}\right)$ we have
 \begin{equation*}
   \begin{aligned}
     &\Big\|\mathbb{E}y_1^j - \frac{\sum_{j=1}^{N}y_1^j}{N}\Big\|^2 = \bigg\|\frac{1}{N}\sum_{i=1}^{N} \Big(\sum_{k=1}^{N}\Lambda_{ki}\Big) (\tilde{x}_i - \mathbb{E}\tilde{x}_i)\bigg\|^2
     \leq L \bigg\|\frac{1}{N}\sum_{i=1}^{N} E (\tilde{x}_i - \mathbb{E}\tilde{x}_i)\bigg\|^2\\ =& L\bigg\| E (\tilde{x}^{(N)} - \hat{x})\bigg\|^2 = O\big(\frac{1}{N}\big).
   \end{aligned}
 \end{equation*}
Thus, next step is to investigate $\mathbf\Lambda$.

 We can verify that $\Lambda_{ii}=\Lambda_1$ for $i=1,\cdots,N$ and $\Lambda_{ij}=\Lambda_2$ for $i\neq j$, where $\Lambda_1$ and $\Lambda_2$ satisfy
\begin{equation*}\resizebox{\linewidth}{!}{$
\left\{
 \begin{aligned}
 		&\dot{\Lambda}_1 \!+\! \Lambda_1(A \!+\! B\Theta_1) \!+\! \frac{(N\!-\!1)\Lambda_2 \!+\! \Lambda_1}{N}F \!+\! Q \!+\! A^T\Lambda_1 \!-\! \frac{1}{N}C^T\Lambda_1\tilde{F} \!-\! C^T\Lambda_1(C \!+\! D\Theta_1) \!=\! 0,\\
 &\dot{\Lambda}_2 \!+\! \Lambda_2(A \!+\! B\Theta_1) \!+\! \frac{(N\!-\!1)\Lambda_2 \!+\! \Lambda_1}{N}F \!+\! A^T\Lambda_2 \!-\! \frac{1}{N}C^T\Lambda_1\tilde{F} = 0,\\
 &\Lambda_1 (T) = G,\quad \Lambda_2(T) = 0,
 \end{aligned}
 \right.$}
\end{equation*}is a solution of $\mathbf\Lambda$,

Hence, $\sum_{k=1}^{N}\Lambda_{ki} = {\Lambda}_1 + (N-1)\Lambda_2$ and we study the uniform boundness of $\Lambda_1$ and $\Lambda_2^*\triangleq (N-1)\Lambda_2$ w.r.t $N$ and $t$, where
\begin{equation}\label{Lambda1 Lambda2*}
\left\{
 \begin{aligned}
 		&\dot{\Lambda}_1 + \Lambda_1\bigg(A + B\Theta_1 + \frac{F}{N}\bigg) + A^T\Lambda_1 - C^T\Lambda_1(C + D\Theta_1 + \frac{\tilde{F}}{N}) + \frac{\Lambda_2^*}{N}F + Q = 0,\\
 &\dot{\Lambda}_2^* + \Lambda_2^*\bigg(A + B\Theta_1 + \frac{N - 1}{N}F\bigg) + A^T\Lambda_2^* + \frac{N-1}{N}(\Lambda_1F - C^T\Lambda_1\tilde{F}) = 0,\\
 &\Lambda_1 (T) = G,\quad \Lambda_2^*(T) = 0.
 \end{aligned}
 \right.
\end{equation}
Firstly, by the linearity, for any given $N$, $\Lambda_1$ and $\Lambda_2^*$ are solvable on $[0,T]$. Let $L$ = $\sup_{0\leq t\leq T}\{\|A(t)\|_{\max}$, $\|B(t)\Theta_1(t)\|_{\max}$, $\|F(t)\|_{\max}$, $\|C(t)\|_{\max}$, $\|D(t)\Theta_1(t)\|_{\max}$, $\|\tilde{F}(t)\|_{\max}$, $\|Q(t)\|_{\max}\}$.
Introduce $\bar{\Lambda}_1$ and $\bar{\Lambda}_2^*$ satisfying
\begin{equation}
\left\{
 \begin{aligned}
 		&\dot{\bar{\Lambda}}_1 + 3L\bar{\Lambda}_1E + LE\bar{\Lambda}_1 + 3LE\bar{\Lambda}_1E + LE\bar{\Lambda}_2^* + LE = 0,\quad \bar{\Lambda}_1 (T) = G,\\
 &\dot{\bar{\Lambda}}_2^* + 2L\bar{\Lambda}_2^*E + LE\bar{\Lambda}_2^* +L(\bar{\Lambda}_1E + E\bar{\Lambda}_1E) = 0,\quad \bar{\Lambda}_2^*(T) = 0.\\
 \end{aligned}
 \right.
\end{equation}
By the linearity,  $\bar\Lambda_1$ and $\bar\Lambda_2^*$ are solvable on $[0,T]$ and surely, bounded. For any $N$ and $t\in[0,T]$, $|\Lambda_1(t)|\leq\bar{\Lambda}_1(t)$ and $|\Lambda_2^*(t)|\leq\bar{\Lambda}_2^*(t)$. Thus, there exist some constants $L$ such that $\displaystyle\sup_{1\leq i\leq N}\sup_{0\leq t\leq T}\|\sum_{k=1}^{N}\Lambda_{ki}(t)\|_{\max} < L$ holds and it follows that $\|\mathbb{E}y_1^j - \frac{\sum_{j=1}^{N}y_1^j}{N}\|^2 = o(1)$. Similarly, $\|\mathbb{E}\beta_1^j - \frac{\sum_{j\neq i}^{N}\beta_1^j}{N}\|^2 = o(1)$ and consequently $\varepsilon_5 = o(1)$.

The estimation of $\varepsilon_6$, which is given by \eqref{Eq24}, follows by Lemma \ref{lemma 11} straightforwardly and $\varepsilon_6 = O\big(\frac{1}{N^2}\big) $. Thus,
\begin{equation*}
   \langle M_2\tilde{\mathbf{u}} + M_1 , \delta\mathbf{u}_i\rangle  = \delta J_i + \sum_{i=1}^{6}\varepsilon_i= O\Big(\frac{1}{\sqrt{N}}\Big).
 \end{equation*}
 Using $\mathbb{E}\int_{0}^{T}\|\delta\mathbf{u}_i\|^2dt<L$ and the Cauchy-Schwartz inequality, $\|M_2\tilde{\mathbf{u}} + M_1\| = O\big(\frac{1}{\sqrt{N}}\big)$.
\end{proof}

\bibliographystyle{plain}

\bibliography{ref}

\
\medskip
Received xxxx 20xx; revised xxxx 20xx.
\medskip

\end{document}